\documentclass[12pt]{amsart}

\usepackage{
amsfonts,
latexsym,
amssymb,
mathabx,
}

\newcommand{\labbel}{\label}

\newcommand{\starr}{\blackdiamond}

\newtheorem{theorem}{Theorem}[section]
\newtheorem{lemma}[theorem]{Lemma}

\newtheorem{proposition}[theorem]{Proposition} 
 
\newtheorem{corollary}[theorem]{Corollary} 
 
\newtheorem*{claim}{Claim} 
\newtheorem*{claim1}{Claim 1} 
\newtheorem*{claim2}{Claim 2} 

\theoremstyle{definition}
\newtheorem{definition}[theorem]{Definition}
\newtheorem{definitions}[theorem]{Definitions}

\newtheorem{problem}[theorem]{Problem}

\theoremstyle{remark}
\newtheorem{remark}[theorem]{Remark}
 
\newtheorem{example}[theorem]{Example}
\newtheorem{examples}[theorem]{Examples}

\newcommand{\brfr}{$\hspace{0 pt}$}

\newcommand{\brfrt}{\hspace{0 pt}}

\DeclareMathOperator{\cf}{cf}

\newcommand{\iit}{\mathfrak{iit}}
\newcommand{\disc}{\mathfrak{d}}
\newcommand{\ord}{\mathfrak{ord}}

\newcommand{\barb}{\pmb} 

\newcommand{\cupdot}{\ensuremath{\mathaccent\cdot\cup}} 

\newcommand{\bigcupdot}{%
\mathop{%
\vphantom{\bigcup}%
\mathpalette\setbigcupdot\cdot}\displaylimits} 
\newcommand{\setbigcupdot}[2]{\ooalign{\hfil$#1\bigcup$\hfil\cr\hfil$#2$\hfil\crcr}}

\begin{document}
 
\title{Ordinal Compactness}

\author{Paolo Lipparini} 
\address{Dipartimentum  Topologi\ae\ (Stic...) \\Viale della Ricerca Scientifica\\II Universit\`a di Roma (Tor Vergata)\\I-00133 ROME ITALY}
\urladdr{http://www.mat.uniroma2.it/\textasciitilde lipparin}



\begin{abstract}
\ We introduce a new  covering property, defined 
in terms of order types of sequences of open sets, rather than in 
terms of  cardinalities of families. 
The most general form of this compactness notion
depends on two ordinal parameters. In the particular
case when the parameters are cardinal numbers, we get back a classical notion.

Generalized to ordinal numbers, this notion
 turns out to behave in a much more varied way.
We prove many nontrivial results of the form 
``every $[  \alpha , \beta ]$-compact space is $[  \alpha', \beta '  ]$-compact'',
for ordinals
 $\alpha$, $ \beta $,  $\alpha'$ and  $ \beta' $, while
only trivial 
results of the above form hold, if we restrict to cardinals. 
Counterexamples are provided showing that our results are optimal.

We present many examples of spaces satisfying the very same 
cardinal compactness properties, but with a broad range of distinct
 behaviors, with respect to ordinal compactness.
A much more refined  theory is obtained for $T_1$ 
spaces, in comparison with arbitrary topological spaces.
The notion of ordinal compactness becomes  partly trivial
for spaces of small cardinality. 
 \end{abstract} 
 
\maketitle

\section{Introduction} \labbel{intro}
 
The nowadays standard notion of \emph{compactness} for topological spaces is usually expressed in terms of cardinalities of open covers, and asserts that every open cover has
a finite subcover. Since compact spaces constitute  a relatively special class, various weakenings   have been extensively considered, the most notable being     \emph{Lindel\"ofness} (``any open cover has a countable subcover''), and 
\emph{countable compactness} (``any countable open cover has a finite subcover'').
Still more generally, \emph{final $ \kappa $-compactness}
asserts that any open cover has a subcover of cardinality $<\kappa$, and  \emph{initial $ \kappa $-compactness}
asserts that every open cover of cardinality $ \leq \kappa$ has a finite subcover.
A vast literature exists on the subject: see the surveys    \cite{Go,St,Vhstt,Vecgt}, and, as a very subjective and partial choice, 
\cite{bella,ST,shts,T} for more recent lines of research. See also the references there.

In this note we extend the notion of cardinal compactness to ordinals, that is, we take into account order types of families of coverings, rather than just their cardinalities.  Assuming the Axiom of Choice, each cardinal can be seen as an ordinal, thus our notion is more general: when a sequence is cardinal-like ordered, we get back the more usual  notions. On the contrary, and quite surprisingly, it turns out that
our ordinal generalization provides a much finer tuning of compactness properties of topological spaces.

\subsection{A first example: Lindel\"of numbers} \labbel{seclind} 
 Before discussing the most general version of our notion,
let us exemplify it in the particular case of  Lindel\"of numbers. 
Let us define the \emph{Lindel\"of\/$^+$  cardinal} of a topological space $X$ as the smallest
\emph{cardinal} $\lambda$ such that every open cover of $X$ has a subcover of cardinality
$ <  \lambda $
(the superscript $^+$ is  a reminder  that the more common definition
 asks just for a subcover of cardinality
$ \leq  \lambda $. The present variant 
is more convenient here, since  it  
distinguishes between compactness and  Lindel\"ofness).
In other words, the Lindel\"of\/$^+$  cardinal of a topological space is the smallest
cardinal $\lambda$ such that the space is finally $ \lambda $-compact.

As an ordinal generalization of the above notion, let us define the \emph{Lindel\"of ordinal}  of a topological space $X$ as the smallest
\emph{ordinal} $\alpha$  such that, for every  open cover of $X$
whose elements are indexed by some ordinal $\beta$, there exists
some subset $H$ of $\beta$ such that $H$ has order type $<\alpha$, and
the set of elements  with index in $H$ still constitutes a cover of $X$.
Thus we are dealing with covers taken in a certain (well) order
and, when dealing with subcovers, we want  the  order  of  the  original  cover to be respected.

While the Lindel\"of ordinal of a space clearly determines its Lindel\"of\/$^+$ cardinal, 
on the contrary, there are spaces with the same Lindel\"of\/$^+$ cardinal, but with very different
Lindel\"of ordinals. 
As a simple example, if $\kappa$ is a regular uncountable cardinal, 
then $\kappa$, both with the discrete topology, and with the order topology,
has Lindel\"of\/$^+$ cardinal $\kappa^+$. On the other hand, though $\kappa^+$ 
is also the Lindel\"of ordinal of the former space,
the latter space  has  a much smaller Lindel\"of ordinal, that is,  $\kappa+ \omega $
(here and below, $+$  denotes \emph{ordinal sum}). Intermediate 
cases can occur: for example, 
the disjoint union of two copies of $\kappa$ with the order topology
has Lindel\"of ordinal
$ \kappa + \kappa + \omega $. 
We can also have $\kappa+1$, $\kappa+2$, \dots \ as Lindel\"of ordinals,
but only in some pathological cases, and only for spaces satisfying very few
separation properties. 
More involved examples shall be presented in the body of the paper.
Thus our  ordinal generalization can be used to distinguish among
spaces which appear to be quite similar, as far as the cardinal notion is considered.

Imposing further conditions on a space provides some constraints on its
Lindel\"of ordinal. For example, the Lindel\"of ordinal of a countable
space is either $ \omega_1$, or is $ \leq \omega \cdot \omega $.  
For spaces of cardinality $\kappa$, there are similar  limitations,
slightly more involved.
Stronger restrictions are obtained by imposing mild separation axioms.
For example, the Lindel\"of ordinal of a $T_1$ space (of any cardinality)
 is either $ \leq\omega$, or
$ \geq \omega _1$.  
Actually, only ordinals of a very special form can both have cofinality 
$ \omega$ and be the Lindel\"of ordinal
of some $T_1$ space (Corollary \ref{lindt1}).
We also show that, for arbitrary spaces, the Lindel\"of ordinal of a disjoint union is
exactly determined by the Lindel\"of ordinals of the summands.

Summing up, the Lindel\"of ordinal of a topological space appears to
be a quite fine measure of the compactness properties the space satisfies.
Moreover, there are interesting and deep connections
between the possible values the Lindel\"of ordinal can take, and cardinalities
and separation properties of spaces.

\subsection{$[ \mu, \lambda ]$-\brfrt compactness (for cardinals)} \label{lmcpn} 
Now we proceed by considering more general forms of  compactness.
All the (cardinal) compactness properties defined in the first paragraph of this introduction can be unified in a single framework by introducing the following two-cardinals property. 
For cardinals $ \mu \leq \lambda $, a  topological space is said to be 
\emph{$[ \mu, \lambda ]$-\brfrt compact} if and only if every open cover 
by at most $\lambda$ sets has a subcover with $<\mu$ sets.
Thus, for example, compactness is the same as $[ \omega , \lambda ]$-\brfrt compactness, for every cardinal $\lambda$, and Lindel\"ofness is $[ \omega_1 , \lambda ]$-\brfrt compactness, for every cardinal $\lambda$. On the other hand,
countable compactness is $[ \omega , \omega  ]$-\brfrt compactness, and, 
more generally,
initial $\lambda$-compactness 
is $[ \omega , \lambda  ]$-\brfrt compactness.

With a restriction on regular cardinals,
and also in various equivalent forms,
the above two-cardinals version has been introduced in 1929 by 
P. Alexandroff and  P. Urysohn
\cite{AU}. 
 For arbitrary cardinals, the very exact form of the above definition
 seems to have first appeared in \cite{Sm}.
It has been studied by many people, sometimes  under different names
and notations, and in several 
equivalent formulations. See a survey of further related notions and results in \cite{Vfund}.
 See also, e. g., 
\cite{G,LO,Vlnm} and references there for further information.

Apart from intrinsic interest, $ [ \mu, \lambda ]$-\brfrt compactness
 has proved useful in many cases. Besides  providing a common generalization 
of  countable compactness, Lindel\"ofness, and so on, 
it exhibits a very interesting feature:
 $[ \mu, \lambda ]$-\brfrt compactness  is equivalent to $[ \nu, \nu ]$-\brfrt compactness, for every $\nu$ with $\mu \leq \nu \leq \lambda $. 
In particular, (full) compactness is equivalent to
$[ \nu, \nu ]$-\brfrt compactness for every infinite cardinal $\nu$,
and Lindel\"ofness is equivalent to
$[ \nu, \nu ]$-\brfrt compactness for every $\nu > \omega $.
 In other words,
 we can  ``slice''  compactness into smaller pieces.
This fact  has found many applications,
mainly in view of the fact that, for $\nu$ regular, 
$[ \nu, \nu ]$-\brfrt compactness has many
equivalent formulations, most notably in terms
of the existence of accumulation points
of sets of cardinality $\nu$. See \cite{Vlnm}.
 See also \cite{topappl},
where the above mentioned ``slicing'' procedure
has found another substantial application. 

By the way, let us also mention that the notion of
$[ \mu, \lambda ]$-\brfrt compactness has ostensibly inspired some further notions
outside  mainstream general topology. Most notably, some of the
earliest definitions of both weakly and strongly compact cardinals  were introduced as forms
of $[ \kappa , \kappa  ]$-\brfrt compactness
for certain infinitary languages \cite[Chapters 17 and 20]{J}. The exact topological content
of these definitions  later clearly emerged: see \cite{C,Ma}  for history, references,
and for other notions in Model Theory and Logic which have apparently
been inspired by $[ \mu, \lambda ]$-\brfrt compactness. Also the notion
of a $( \mu, \lambda )$-\brfrt regular ultrafilter, which played some role
in the evolution of 
Set Theory \cite{CN}, \cite[Section 13]{KM}, \cite[p. 373]{J}, apparently originated in this stream of ideas.

\subsection{The ordinal generalization} \label{secord} 
Motivated by the interest of (cardinal) $[ \mu, \lambda ]$-\brfrt compactness,
we started considering the possibility of an ordinal generalization. 
Though initially misled by the observation that
``initial $\alpha$-compactness''  actually reduces to 
a cardinal notion (Corollary \ref{initial}), we soon realized that the more general
notion of  ``two  ordinals compactness''  is really new, as exemplified 
above in the particular case of Lindel\"of-like properties or, put in other words, 
final $\alpha$-compactness.

In detail, if $\beta$ and $\alpha$ are ordinals, let us say that a space $X$  
is $[ \beta , \alpha ]$-\brfrt compact
if and only if 
every $\alpha$-indexed open cover has a 
subcover indexed by a set of order type $<\beta$
(in the induced order).

Ordinal compactness, in the above sense, turns out to have 
some very particular features. As in the case of cardinals,
we can show that, also for ordinals,  $[ \beta , \alpha ]$-\brfrt compactness 
is equivalent to $[ \gamma , \gamma  ]$-\brfrt compactness, for every ordinal 
$\gamma$ with $ \beta \leq \gamma \leq \alpha$. However, the similarities
essentially stop here. Indeed, for $\mu \not=\lambda$ infinite regular cardinals, $[ \mu, \mu]$-\brfrt compactness and $[ \lambda , \lambda ]$-\brfrt compactness are independent properties.
On the other hand, for 
 ordinals, we have many
results  which tie together 
$[ \beta , \alpha ]$-\brfrt compactness
and $[ \beta' , \alpha' ]$-\brfrt compactness,
for various $\beta$, $\alpha$, $\beta'$ and $\alpha'$. 
Just to state some of the simplest relations, we have that, for $\alpha$ and $\beta$ infinite ordinals,
  \begin{enumerate}  
\item
If $\beta \leq \alpha $, then $[ \beta  , \alpha  ]$-\brfrt compactness
implies 
$[ \beta  , \alpha +1 ]$-\brfrt compactness.
\item
$[ \beta + \alpha  , \beta + \alpha  ]$-\brfrt compactness
implies 
$[ \beta + \alpha + \alpha  , \beta + \alpha + \alpha  ]$-\brfrt compactness.
\item
$[ \alpha  , \alpha  ]$-\brfrt compactness
implies both
$[ \beta + \alpha  , \beta + \alpha  ]$-\brfrt compactness
and 
$[ \beta \cdot \alpha  , \beta \cdot \alpha  ]$-\brfrt compactness.
  \end{enumerate} 
However, not ``everything'' is provable, even for ordinals having the same cardinality. Indeed, still presenting only some simple examples:
  \begin{enumerate} 
   \item[(4)]   
$[ \alpha +1 , \alpha +1 ]$-\brfrt compactness
does not imply 
$[ \alpha  , \alpha  ]$-\brfrt compactness, in general.
 \item[(5)]   
$[ \kappa + \omega  , \kappa + \omega  ]$-\brfrt compactness
does not imply 
$[ \kappa   , \kappa  ]$-\brfrt compactness, in general.
\item[(6)]  
$ [ \kappa  + \kappa , \kappa  + \kappa  ] $-\brfrt compactness 
does not imply
$ [ \kappa \cdot \kappa , \kappa  \cdot \kappa ] $-\brfrt compactness,
in general.
\end{enumerate}

Thus, ordinal compactness is a highly nontrivial notion,
 in comparison with cardinal compactness.
Moreover, the ordinal compactness properties
of a topological space are deeply affected both by  its cardinality and 
its  separation properties. For example, 
for $\kappa$ an infinite regular cardinal, any counterexample to 
Clause (6) above must be of cardinality $>\kappa$.
On the other hand, no $T_1$ space can be a counterexample to Clause (4).
Considering the compactness properties of disjoint unions
involves some problems on ordinal arithmetic which are not completely trivial.

 $T_1$ spaces turn
out to be a somewhat neat dividing line: many rather odd counterexamples,
 possible in spaces lacking separation properties, cannot be constructed 
using  $T_1$ spaces. 
Thus we provide a quite neat theory for $T_1$  spaces. 
In particular, in this respect, countable ordinals
behave very differently from uncountable ones.
The compactness theory for $T_1$ spaces is trivial on countable ordinals;
more generally,
 apart from a few exceptions, the ordinal properties
of a $T_1$ space are ``invariant'' modulo intervals of countable length. 
Apparently, assuming stronger separation axioms does not seem to modify the theory a lot; at large, we get
essentially the same results and counterexamples for $T_1$ and for normal spaces.
However, there is still room for the possibility of some finer results 
holding only for normal spaces; this is left as an open problem.

\subsection{Synopsis of the paper} \label{summary} 
In summary, the paper is divided as follows. In Section \ref{basic} 
we introduce the main definition, together with some relatively simple properties and a couple of equivalent reformulations. Then we prove many results of the form
``every $[  \alpha , \beta ]$-compact space is $[  \alpha', \beta '  ]$-compact''; 
most of these results shall be used in the rest of the paper.
In Section \ref{examples} we then provide a lot of examples, showing
that $[ \beta , \alpha ]$-\brfrt compactness, for $\alpha$ and $\beta$ ordinals, 
provides a very fine tuning of  properties of open coverings: there are many spaces
which show a very differentiated behavior with respect to ordinals, but behave exactly the same way, when $\alpha$ and $\beta$ are taken to vary only on cardinals.
We also show that many of the results of Section \ref{basic}
are the best possible ones. The most basic examples are presented in Subsection \ref{basicex}; then in Subsection \ref{disgsec} we  discuss the behavior of ordinal compactness with respect to disjoint unions, and show that many more counterexamples can be obtained in such a way. We also introduce 
a generalized form of infinite disjoint union with a partial compactification.
Compactness properties of disjoint unions are shown to be connected to some notions
in ordinal arithmetics related to  natural sums of  ordinals. 
Such matters are clarified in detail in Subsection 
\ref{rmkonsh}.   
 
In Section \ref{small} we show that many more implications 
between compactness properties hold, for spaces of small cardinality; put
in another way, certain counterexamples can be constructed only by means of spaces of
sufficiently large cardinality. Such counterexamples are indeed provided in 
Section \ref{exact}, where we give an exact characterization
of those pairs of ordinals $\alpha$ and $\beta$ such that 
$[ \alpha, \alpha ]$-\brfrt compactness implies
$[ \beta, \beta ]$-\brfrt compactness.
In Section \ref{t1sec} we then get a more refined theory, which holds
for $T_1$ spaces. For such spaces,
$[ \beta , \alpha ]$-\brfrt compactness becomes trivial for countably infinite ordinals
(Corollary \ref{t1cor}).
More generally, with a few exceptions,  ordinal compactness  for $T_1$ spaces is invariant modulo
intervals of countable length. Finally, Section 
\ref{concl} contains  various quite disparate remarks and problems. In particular, it introduces further generalizations of ordinal compactness, and also discusses the possibility of 
a variant in a model theoretical sense. 

\smallskip

The present note by no means exhausts all that can be said about $[ \beta, \alpha ]$-\brfrt compactness. Furthermore, as we mentioned, the notion of $[ \beta, \alpha ]$-\brfrt compactness can be also generalized to different contexts.

\section{Main definition and basic properties} \labbel{basic} 

In this section we introduce our main notion, and state some simple properties.
We compare it with the more usual notion which deals only with cardinals; then
we start proving results of the form ``every $[ \beta, \alpha ]$-\brfrt compact space is
$[ \beta', \alpha' ]$-\brfrt compact'', for appropriate ordinals. For cardinal compactness, only trivial results of the above kind hold. In the subsequent sections we shall present counterexamples showing  that
our results cannot be improved.

Throughout, let $\alpha$, $\beta$ and $\gamma$ be nonzero ordinals, 
and $\lambda$, $\mu$ be nonzero cardinals. As custom,
we shall assume the Axiom of Choice, hence we can identify cardinals 
with initial ordinals.

\begin{definition} \labbel{ordcpn}
If $X$ is a nonempty set (usually, but not necessarily, a topological space),
and $\tau$ is a nonempty family of subsets of $X$, we say that 
$(X, \tau)$ is \emph{$[ \beta, \alpha ]$-\brfrt compact} if and only if the following condition holds.

Whenever $(O_ \delta ) _{ \delta \in \alpha } $ is a sequence of 
members of $\tau$ such that 
$ \bigcup  _{ \delta \in \alpha } O_ \delta = X $,
then there is $H \subseteq  \alpha $ with
order type $<\beta$ and such that
$ \bigcup  _{ \delta \in H } O_ \delta = X $.

If there is no danger of confusion, we shall  simply say $X$ in place of
$(X, \tau)$. As usual, a sequence $(O_ \delta ) _{ \delta \in \alpha } $  of 
members of $\tau$ such that 
$ \bigcup  _{ \delta \in \alpha } O_ \delta = X $ shall be called 
a \emph{cover} of $X$. A \emph{subcover} of  $(O_ \delta ) _{ \delta \in \alpha } $ is a subsequence which itself is a cover.

By
\emph{$[ \beta, \alpha )$-\brfrt compactness} we mean
$[ \beta, \alpha' ]$-\brfrt compactness
for all $\alpha' < \alpha $. The notation is justified by 
Proposition \ref{simple}(4) below. 
Another notation for $[ \beta, \alpha )$-\brfrt compactness is
$[ \beta, < \alpha ]$-\brfrt compactness.

Finally, \emph{$[ \beta, \infty)$-\brfrt compactness} is
$[ \beta, \alpha ]$-\brfrt compactness for all ordinals $\alpha \geq \beta $.\end{definition}   

When $\alpha$ and $\beta$ are both cardinals, and $X$ 
is a topological space ($\tau$ being always understood to be the topology on $X$),
we get back the classical cardinal  compactness notion of Alexandroff,  Urysohn and Smirnov \cite{AU,Sm}.
This is because, for $\lambda$ a cardinal, having order type $<\lambda$ is the 
same as having cardinality $<\lambda$. 

Notice that we allow repetitions in 
$(O_ \delta ) _{ \delta \in \alpha } $, that is,
we allow the possibility that
$O_ \delta = O _{ \delta '}  $, for
$\delta \not = \delta' $.
An equivalent and  sometimes useful
definition in which (among other things) repetitions 
are not allowed is given by Lemma \ref{irredundant}.
We have given the  definition in the present form since it appears somewhat simpler.

\begin{remark} \labbel{ordset}
In the definition of $[ \beta, \alpha ]$-\brfrt compactness,
the assumption that the sequence is indexed by elements in the ordinal $\alpha$
is only for convenience. We get an equivalent definition by asking that,
for every well ordered set $J$ of order type $\alpha$,
if   
$(O_ j ) _{ j \in J } $ is a cover of $X$,
then there is $H \subseteq J$ such that the order type
of $H$ (under the order induced by the order on $J$) is 
$< \beta$, and such that
$(O_ j ) _{ j \in H } $ is a cover of $X$.
 \end{remark}   

Of course, $[ \beta, \alpha ]$-\brfrt compactness is equivalent to the following condition (just take complements!).
Whenever $(C_ \delta ) _{ \delta \in \alpha } $ is a sequence of 
complements of members of $\tau$,
and $ \bigcap  _{ \delta \in H } C_ \delta \not= \emptyset  $,
for every  $H \subseteq  \alpha $ with
order type $<\beta$,
then 
 $ \bigcap  _{ \delta \in \alpha } C_ \delta \not= \emptyset $.

As we shall see below in Remarks \ref{w} and \ref{samecard}, ordinal compactness is actually 
a new notion, that is,  it cannot be defined in terms of 
cardinal compactness.

 We first list some simple but useful properties of 
$[ \beta, \alpha ]$-\brfrt compactness.

\begin{proposition} \labbel{simple}
Let $\alpha$ and $\beta$ be nonzero ordinals.
\begin{enumerate} 
\item If $\beta \leq \beta '$ and $ \alpha' \leq \alpha$ then
$[ \beta, \alpha ]$-\brfrt compactness implies $[ \beta', \alpha' ]$-\brfrt compactness. 
\item 
$[ \beta, \alpha ]$-\brfrt compactness is equivalent to $[ \gamma , \gamma  ]$-\brfrt compactness
for every $\gamma$ with $\beta \leq \gamma \leq \alpha $. 
\item
If $ \beta \leq \beta' \leq \alpha $, then  $X$ is
$[ \beta, \alpha ]$-\brfrt compact if and only if $X$ is 
both  $[ \beta, \beta ' )$-\brfrt compact and $[ \beta', \alpha ]$-\brfrt compact.
\item 
$[ \beta, \alpha )$-\brfrt compactness is equivalent to $[ \gamma , \gamma  ]$-\brfrt compactness
for every $\gamma$ with $\beta \leq \gamma < \alpha $. 
\end{enumerate}  
 \end{proposition} 

 \begin{proof} 
(1) is trivial. If $\alpha' < \alpha $, add dummy elements at the top of  the sequence, for example, by adding new occurrences of one element already in the sequence.

One implication in (2) is immediate from (1). 
 
The converse is obtained by transfinite induction.
Suppose that $X$ is $[ \gamma , \gamma  ]$-\brfrt compact,
for every $\gamma$ with $\beta \leq \gamma \leq \alpha $.
We shall prove $[ \beta  , \gamma  ]$-\brfrt compactness,
for every $\gamma$ with $\beta \leq \gamma \leq \alpha $, by induction on $\gamma$.
The induction basis $\gamma= \beta $ is true by assumption.
As for the induction step, let $ \beta  < \gamma \leq \alpha $, and assume 
that $X$ 
is $[ \beta  , \gamma'  ]$-\brfrt compact, for every $\gamma'$ with $ \beta \leq \gamma' < \gamma $. 
 Let $(O_ \delta ) _{ \delta \in \gamma  } $ be a cover of $X$.
By $[ \gamma , \gamma  ]$-\brfrt compactness, $(O_ \delta ) _{ \delta \in \gamma } $
has a subcover $\mathcal S$ whose index set has  order type $ \gamma' <\gamma$. If
 $ \gamma '< \beta $, we are done. Otherwise, by 
$[ \beta  , \gamma'  ]$-\brfrt compactness, 
and Remark \ref{ordset}, we get a subcover of $\mathcal S$ whose index set has order type $< \beta $, and
the item is proved.    

(3) The only if  condition is immediate from (1).
For the converse,  notice that, again by (1),
$[ \beta, \beta ' )$-\brfrt compactness implies 
$[ \gamma , \gamma  ]$-\brfrt compactness
for every $\gamma$ with $\beta \leq \gamma < \beta'  $,
 and that $[ \beta', \alpha ]$-\brfrt compactness 
implies 
$[ \gamma , \gamma  ]$-\brfrt compactness
for every $\gamma$ with $\beta' \leq \gamma \leq \alpha   $.

Thus we get $[ \gamma , \gamma  ]$-\brfrt compactness,
for every $\gamma$ with $\beta \leq \gamma \leq \alpha   $,
hence $[ \beta, \alpha ]$-\brfrt compactness, by (2).

(4) is immediate from (2).
\end{proof} 

\begin{remark} \labbel{ordcard} 
When $\alpha$, $\beta$, $\alpha'$ \dots are restricted to vary only on cardinals,
 rather than ordinals, Proposition \ref{simple} still holds, with the same proof. In fact, for infinite cardinals, (1) and (2) are classical results about $[ \mu, \lambda  ]$-\brfrt compactness.
Again for infinite cardinals, it is well known (and easy to prove)
that, for topological spaces, $[ \cf \mu, \cf \mu]$-\brfrt compactness
implies $[  \mu,  \mu]$-\brfrt compactness.
An ordinal generalization of the above  fact will be given
in Corollary \ref{cortransfer}(8). 

For infinite regular cardinals, there is no other nontrivial implication
 between $[ \mu , \lambda ]$-\brfrt compactness and
$[ \mu', \lambda' ]$-\brfrt compactness, except for those which follow immediately
from the above mentioned facts. Indeed, if
 $\lambda $ is a 
regular infinite cardinal, then $\lambda$, with the order topology,
is not $[ \lambda , \lambda ]$-\brfrt compact, but it 
is $[ \mu, \mu]$-\brfrt compact for every infinite cardinal 
$\mu \not = \lambda $. Hence, 
if $\mu \leq \mu'$ are infinite cardinals, then 
 $\lambda$ with the order topology
is $[ \mu, \mu']$-\brfrt compact if and only if $ \lambda \not\in [ \mu, \mu']$.
(Here and in what follows
$[ \mu, \mu']$ shall denote the interval consisting of those ordinals $\delta$ such that 
$\mu \leq \delta \leq \mu '$.) 
More generally, the exact ordinal compactness properties of $\lambda$ (with various  topologies)
shall be determined in 
Example \ref{exex}. 

Contrary to the case of cardinal compactness, and quite surprisingly,  there are many nontrivial
``transfer properties'' for ordinal compactness, relating $[ \beta, \alpha ]$-\brfrt compactness and
$[ \beta', \alpha' ]$-\brfrt compactness, for various 
$\beta$, $\alpha$, $\beta'$ and $\alpha'$.
The next proposition and its corollary list some simple relations.
More significant results along this line, and some characterizations
 shall be proved in Section \ref{exact}.
\end{remark} 

\begin{proposition} \labbel{transfer}
Suppose that $\beta$, $\alpha$, $\beta'$ and $\alpha'$
are nonzero ordinals, and that there exists an injective function
$f: \alpha ' \to \alpha $ 
such that, for every  $K \subseteq \alpha $ with order type 
$<\beta$, it happens that $f ^{-1} (K)$ has order type $< \beta' $.

Then  $[ \beta, \alpha ]$-\brfrt compactness implies
$[ \beta', \alpha' ]$-\brfrt compactness.

The assumption that $f$ is injective can be dropped
in the case of topological spaces (or just assuming that $\tau$ is  closed under  unions).
\end{proposition}  

\begin{proof}
Suppose that $(X, \tau )$ is $[ \beta, \alpha ]$-\brfrt compact, and let $f$ be given
satisfying the assumption.
Let $(O_ \delta ) _{ \delta \in \alpha' } $ be a cover of $X$,
and let $(U_ \varepsilon  ) _{ \varepsilon  \in \alpha } $ 
be defined by 
$U _ \varepsilon = O_ \delta  $, if $f( \delta )= \varepsilon $,
and arbitrarily, if $\varepsilon$ is not in the image of $f$.
The definition is well posed, since $f$ is injective.
Let $\alpha''$ be the order type of $f(\alpha')$  

$(U_ \varepsilon  ) _{ \varepsilon  \in f(\alpha') } $ is still a cover of $X$, hence, 
by $[ \beta, \alpha'' ]$-\brfrt compactness
 (which follows from 
$[ \beta, \alpha ]$-\brfrt compactness, 
by Proposition \ref{simple}(1)),
and by Remark \ref{ordset},  
there is $K \subseteq \alpha $
of order type $<\beta$ and such that  
$(U_ \varepsilon  ) _{ \varepsilon  \in K } $ still
covers $X$. If we put
$H=f ^{-1}(K) $, then, by assumption, 
$H$ has order type $<\beta'$;
moreover, $(O_ \delta ) _{ \delta \in H } $
is a cover of $X$, hence 
$[ \beta', \alpha' ]$-\brfrt compactness is proved.

In case $\tau$ is closed under unions, and $f$ is not injective, 
define
$U _ \varepsilon = \bigcup _{f( \delta )= \varepsilon } O_ \delta  $,
and the same argument carries over.
 \end{proof}

In what follows,  if not otherwise specified, the operation $+$ will 
  denote \emph{ordinal sum}. That is, $\alpha+ \beta $  is the order type of  
the order obtained by attaching a copy of $\beta$ ``at the top''  of  $\alpha$.
Similarly, $\cdot$ denotes \emph{ordinal product}. 

The next corollary provides a sample of results
that can be proved about the relationship between
$[ \beta,   \alpha  ]$-\brfrt compactness,
and $[ \beta',  \alpha'  ]$-\brfrt compactness,
for various ordinals. Most of them shall be used  in the rest of the paper.

\begin{corollary} \labbel{cortransfer}
Supose that $\alpha$, $\beta$ and $\gamma$ are nonzero ordinals, and
$\lambda$, and $\nu$  are cardinals. 
\begin{enumerate}
\item 
If $\beta \leq \alpha $, 
and $\alpha$ is infinite, 
then
$[ \beta, \alpha ]$-\brfrt compactness implies
$[ \beta, \alpha + 1 ]$-\brfrt compactness, 
hence also $[ \beta, \alpha + n ]$-\brfrt compactness, 
for each $ n < \omega$.
\item 
If either $\gamma$ or $\alpha$ is infinite, then
$[ \gamma +\alpha, \gamma +\alpha ]$-\brfrt compactness implies
$[ \gamma +\alpha   +\alpha, \gamma + \alpha + \alpha  ]$-\brfrt compactness, hence 
also $[ \gamma +\alpha \cdot n  , \gamma + \alpha \cdot n  ]$-\brfrt compactness,
for each $ n < \omega$.
\item
If $ \beta \leq \alpha $, $\alpha$ is infinite, and $\lambda= \cf \alpha $, then  $[ \beta , \alpha ]$-\brfrt compactness implies
$[ \beta  , \alpha + \lambda \cdot \omega )  $-\brfrt compactness.
\item 
  If   $ \beta \leq\lambda $, 
then 
$[ \beta, \lambda  ]$-\brfrt compactness implies
$[ \beta, \lambda ^+ )$-\brfrt compactness.
\item
If $ \beta \leq \alpha  + \lambda $, 
and  either $ \cf \alpha > \lambda $, 
or $\alpha$ can be written as a limit of ordinals of
cofinality $> \lambda $,  
then 
$[ \beta, \alpha  + \lambda  ]$-\brfrt compactness implies
$[ \beta,   \alpha + \lambda ^+) $-\brfrt compactness.
  \end{enumerate} 
Suppose further that $\tau$ is closed under unions. Then:
  \begin{enumerate}     \setcounter{enumi}{5}
 \item
$[\alpha , \alpha ]$-\brfrt compactness implies $[ \beta +\alpha , \beta +\alpha ]$-\brfrt compactness. 
\item 
$[\alpha , \alpha ]$-\brfrt compactness implies 
$[ \beta \cdot \alpha , \beta \cdot \alpha ]$-\brfrt compactness. 
\item 
If  $\cf \alpha = \nu$ is infinite,
then
$[\nu , \nu ]$-\brfrt compactness implies 
$[  \alpha ,  \alpha ]$-\brfrt compactness. 
 \end{enumerate}  
 \end{corollary} 

\begin{proof} 
(1) In view of Proposition \ref{simple}(1), 
and since $\beta \leq \alpha $,
$[ \beta, \alpha ]$-\brfrt compactness implies
$[ \alpha , \alpha  ]$-\brfrt compactness.
In view of Proposition \ref{simple}(3),  it is then enough to show that 
$[ \alpha , \alpha  ]$-\brfrt compactness
implies $[ \alpha +1, \alpha+1  ]$-\brfrt compactness.

The latter is proved by applying
 Proposition \ref{transfer} to the function $f: \alpha + 1 \to \alpha $
 defined as follows.
\[ 
f( \varepsilon )=\begin{cases}
0&    \text{if  $ \varepsilon = \alpha $},\\
\varepsilon + 1&    \text{if  $ \varepsilon < \omega$},\\
\varepsilon &    \text{if  $\omega \leq \varepsilon < \alpha  $}.\\
\end{cases}
\] 

(2) If $\alpha$ is finite, then  $\gamma$ is infinite, and the result follows from (1).

Otherwise, suppose that $\alpha= \alpha ' + n$, with $\alpha'$ limit
and $ n<\omega$.
Thus $\gamma+ \alpha + \alpha = \gamma + \alpha ' + \alpha ' + n $. 
Consider the following
function $f: \gamma + \alpha + \alpha  \to \gamma  + \alpha    $.
\[ 
f( \varepsilon )=\begin{cases}
\varepsilon &    \text{if  $ \varepsilon < \gamma  $},\\
\gamma + 2 m&    \text{if  $ \varepsilon = \gamma + m $, with $m \in \omega $},\\
\gamma + \alpha '' + 2 m &    \text{if  $ \varepsilon = \gamma + \alpha'' +m $, with 
$ \alpha '' $ limit $ < \alpha' $,   $m \in \omega $},\\
\gamma + 2 m+1&    \text{if  $ \varepsilon = \gamma + \alpha' +m $, with $m \in \omega $},\\
\gamma + \alpha '' + 2 m +1&    \text{if  $ \varepsilon = \gamma + \alpha' + \alpha'' +m $, } \text{with 
$ \alpha '' $ limit $ < \alpha' $,  $m \in \omega $},\\
\gamma + \alpha ' + m  &    \text{if  $ \varepsilon = \gamma + \alpha' + \alpha' +m $, with 
 $m < n$}.\\
\end{cases}
\] 
It is easy to see that $f$ is injective. 

Suppose that 
$K \subseteq  \gamma  + \alpha = \gamma + \alpha ' + n$, and $K$  has order type 
$<  \gamma  + \alpha $.  Then either 
(a)
$K \cap [ \gamma + \alpha' , \gamma +\alpha'+n) $ has order type $ <n$,
or   (b)
$K \cap [ \gamma , \gamma +\alpha') $ has order type $ \alpha ^*< \alpha' $,
or 
(c) 
$K \cap \gamma$ has order type $ \gamma ^*< \gamma$.

If (a) holds, then $f ^{-1} ([ \gamma + \alpha' ,\gamma + \alpha ' + n))$
has order type $<n$, hence  $f ^{-1}(K) $ has order type 
$< \gamma + \alpha' + \alpha ' + n = \gamma + \alpha + \alpha $.  In case (b),
$f ^{-1} ([ \gamma ,\gamma + \alpha '))$
has order type $ \leq \alpha ^* + \alpha ^* $,
hence $f ^{-1}(K) $ has order type 
$ \leq \gamma + \alpha ^* + \alpha ^* + n$,
which is strictly smaller than 
  $\gamma + \alpha' + \alpha ' + n$, since 
$ \alpha ^*< \alpha' $.
Finally, we can suppose that we are in case (c), and both (a) and (b)
fail. 
Since $K$ has order type 
$<  \gamma  + \alpha = \gamma + \alpha ' +n$ and
$K \cap \gamma$ has order type $ \gamma ^*< \gamma$,
 then $\gamma^* + \alpha < \gamma + \alpha $.
This easily implies that
$\gamma^* + \alpha + \alpha < \gamma + \alpha + \alpha $
(for example, by expressing $\gamma^*$, $\gamma$ and $\alpha$ in Cantor normal form).
Since $f$ is injective and, restricted to $\gamma$, is the identity, then  
$f ^{-1}(K) $ has order type 
$\leq \gamma^* + \alpha + \alpha < \gamma + \alpha + \alpha $.

We have proved that 
$f ^{-1}(K) $ has order type 
$< \gamma + \alpha + \alpha $
in all cases, hence  Proposition \ref{transfer} can be applied. 

(3) 
 If $ \cf \alpha = 1$, this follows from (1), hence let us suppose that 
$ \cf \alpha \geq \omega $.

By Proposition \ref{simple}, it is enough to prove that 
 if $\delta < \lambda \cdot \omega  $, then  $[ \alpha, \alpha ]$-\brfrt compactness implies
$[ \alpha + \delta , \alpha + \delta  ]$-\brfrt compactness. 
Refining further, it is enough to prove that 
\begin{equation} \tag{*}    
\text{if 
$\delta \leq \lambda   $, } 
\\
\text{then  $[ \alpha, \alpha ]$-\brfrt compactness
implies
$[ \alpha + \delta , \alpha + \delta  ]$-\brfrt compactness,}
\end{equation}
 since then
$[ \alpha, \alpha ]$-\brfrt compactness implies
$[ \alpha + \lambda  , \alpha + \lambda   ]$-\brfrt compactness, and then we can proceed inductively, by applying the result with $\alpha + \lambda   $ in place of 
$\alpha$, and then  with $\alpha + \lambda  + \lambda $ in place of 
$\alpha$, and so on.

Hence, suppose that $\delta \leq \lambda =\cf \alpha   $, and that  $[ \alpha, \alpha ]$-\brfrt compactness holds. If $\delta + \alpha > \alpha $, then necessarily 
$\delta= \cf \alpha $ and $ \alpha = ( \cf \alpha) \cdot m$, for some
$ m < \omega$, and (*) follows from (2) with $\gamma=0$. 
Otherwise,  $\delta + \alpha = \alpha $, hence we can define the following 
injective function
$f: \alpha + \delta \to \alpha $.
\[ 
f( \varepsilon )=\begin{cases}
\delta + \varepsilon  &    \text{if  $ \varepsilon < \alpha  $},\\
\eta &    \text{if  $ \varepsilon = \alpha + \eta $, for $\eta < \delta $}.\\
\end{cases}
\] 
Now, if $K \subseteq \alpha$  has order type $ \zeta < \alpha$,
then $f ^{-1} (K)$   has order type $\leq  \zeta + \delta $, 
which is necessarily
 $< \alpha + \delta $,
since  $\delta  \leq \cf \alpha $.
Hence Proposition \ref{transfer} can be applied in order to get (*).

(4) Again by Proposition \ref{simple},
it is enough to prove that 
$[ \lambda , \lambda ]$-\brfrt compactness implies
$[ \alpha ,  \alpha  ]$-\brfrt compactness,
for every $\alpha$ with $| \alpha |= \lambda $. 
This is accomplished by Proposition \ref{transfer},
letting $f$ be any injection from $\alpha$ to $\lambda$.

(5) As above, it is sufficient to prove
that 
$[ \alpha  + \lambda , \alpha  +\lambda ]$-\brfrt compactness implies
$[ \alpha + \gamma , \alpha + \gamma  ]$-\brfrt compactness,
for every $ \gamma $ with $| \gamma |= \lambda $. 
Let $g$ be any injection from $ \gamma $ to $\lambda$, and apply Proposition \ref{transfer} to the following function
$f:  \alpha + \gamma  \to \alpha  + \lambda $.
\[ 
f( \varepsilon )=\begin{cases}
 \varepsilon &    \text{if  $ \varepsilon < \alpha  $},\\
\alpha + g( \eta )  &    \text{if  $ \varepsilon = \alpha + \eta $, with $\eta < \gamma $}.\\
\end{cases}
\] 

If $K \subseteq \alpha + \lambda $ has order type
 $ <\alpha + \lambda $, then either
$K \cap \alpha $ has order type $  < \alpha $,
or 
$K \cap [ \alpha , \alpha +\lambda) $ has order type $ < \lambda $.
In the latter case, 
and since $\lambda$ is a cardinal, 
we have that $f ^{-1} (K)$ has order type 
$ \leq \alpha + \gamma '$, for some $\gamma'$ 
with $|  \gamma '| < \lambda $,
hence  $f ^{-1} (K)$ has order type 
$ < \alpha + \gamma $, since  $| \gamma |= \lambda $. 

On the other hand, if 
$K \cap \alpha $ has order type $  < \alpha $,
then 
$f ^{-1} (K) \cap \alpha $ has order type $  < \alpha $,
since $f$ is the identity on $\alpha$. The assumptions on $\alpha$,
and $| \gamma |= \lambda $ 
then imply that 
$f ^{-1} (K) $ has order type $  < \alpha + \gamma $.
 
(6) Apply the last statement in 
Proposition \ref{transfer} to the function
 $f: \beta + \alpha \to \alpha  $ defined by 
\[ 
f( \varepsilon )=\begin{cases}
0&    \text{if  $ \varepsilon < \beta  $},\\
\eta  &    \text{if  $ \varepsilon = \beta + \eta $, with $\eta< \alpha $}.\\
\end{cases}
\] 

(7) Apply the last statement in 
Proposition \ref{transfer} to the function
 $f: \beta \cdot \alpha \to \alpha  $ defined by 
$f( \varepsilon )= \zeta $
if  $ \varepsilon = \beta \cdot \zeta + \eta $,
for some $\eta < \beta $.

(8) Let $(\gamma_ \eta) _{ \eta \in \nu} $
be a sequence cofinal in $\alpha$ of order type $\nu$.
Define $f: \alpha \to \nu$ by
$f( \varepsilon )= \inf \{  \eta \in \nu \mid \varepsilon < \gamma _ \eta \} $,
and apply  Proposition \ref{transfer}.
\end{proof}

\begin{example} \labbel{unions}
As suggested by Corollary \ref{cortransfer} (6)-(8), 
the relationships between various ordinal compactness properties
change according to whether  $\tau$ is required or not to be
 closed under unions.
For example, if $ \lambda > \mu$ are infinite 
cardinals, then every $[\mu, \mu]$-\brfrt compact topological space is $[\lambda +\mu,  \lambda +\mu]$-\brfrt compact, by Corollary \ref{cortransfer}(6).
On the other hand, if $X= (\lambda +\mu, \tau )$, where
$ \tau = \{ [0, \beta ) \mid \beta \in \lambda \} 
\cup \{ [ \lambda , \lambda + \gamma  ) \mid \gamma  \in \mu \}$,
 then $X$ is trivially $[\mu, \mu]$-\brfrt compact 
(since it has no cover of cardinality $\mu$), but it is not
$[ \lambda +\mu,  \lambda +\mu]$-\brfrt compact.
This is an example of a more general fact: see Corollary \ref{coriff1}. See
also Example 
\ref{exkk}.
 \end{example}

We shall see in Sections \ref{examples} and 
\ref{exact} 
  that
$[ \beta  , \alpha  ]$-\brfrt compactness
is very far from being a trivial notion. However,
Corollary \ref{cortransfer}(4) implies that
 $[ \beta  , \alpha  ]$-\brfrt compactness becomes partly trivial
for intervals containing a cardinal.

\begin{corollary} \labbel{initial}
If $\alpha$ is infinite, and
 $ \beta \leq |\alpha|$, then the following properties are equivalent.
  \begin{enumerate}  
  \item
 $[ \beta  , |\alpha|  ]$-\brfrt compactness.
   \item
 $[ \beta  , |\alpha|^+  )$-\brfrt compactness.
   \item
 $[ \beta  , \alpha ]$-\brfrt compactness.
  \end{enumerate}
In particular, if $\mu \leq \lambda $ 
are infinite cardinals, then 
 $[ \mu  , \lambda  ]$-\brfrt compactness
is equivalent to 
$[ \mu  , \lambda ^+ )$-\brfrt compactness.
 \end{corollary}

 \begin{proof} 
(1) $\Rightarrow $  (2) is from Corollary \ref{cortransfer}(4). 

(2) $\Rightarrow $  (3) and (3) $\Rightarrow $  (1) are immediate from Proposition 
\ref{simple}(1). 
\end{proof}  

In particular,
``initial $\alpha$-compactness'', that is, $[ \omega , \alpha  ]$-\brfrt compactness,
does become trivial, in the sense that  it actually reduces to cardinal compactness,
in fact, to  $[ \omega   , |\alpha|  ]$-\brfrt compactness.

The next Lemma gives a somewhat
useful  equivalent formulation
of
$[ \beta, \alpha ]$-\brfrt compactness.
It states that
it is enough to take into account
only covers  which are made of  ``irredundant'' elements.

\begin{lemma} \labbel{irredundant}
Let $X$ be a nonempty set,
 $\tau$ be a nonempty family of subsets of $X$, and
$\beta$, $\alpha$ be nonzero ordinals.

Then  
$(X, \tau)$ is $[ \beta, \alpha ]$-\brfrt compact if and only if
the following condition holds.

Whenever $\alpha^*\leq \alpha $, and  $(O_ \delta ) _{ \delta \in \alpha ^*} $ is a sequence of 
members of $\tau$ such that 
\begin{enumerate}
\item
$ \bigcup  _{ \delta \in \alpha^* } O_ \delta = X $, and
\item
for every $\delta <  \alpha ^*$, 
$O_ \delta $ is not contained in 
 $ \bigcup  _{ \varepsilon < \delta  } O_ \varepsilon  $,
 \end{enumerate}   
then there is $H \subseteq  \alpha^* $ with
order type $<\beta$ and such that
$ \bigcup  _{ \delta \in H } O_ \delta = X $.
 \end{lemma}

\begin{proof} 
The ``only if'' part follows trivially from Proposition \ref{simple}(1). 

Conversely, suppose that  $(O_ \delta ) _{ \delta \in \alpha } $
is a cover of $X$.
Let $K=
 \{ \delta \in \alpha \mid O_ \delta  \text{ is not } 
\brfr \text{contained in }
 \bigcup  _{ \varepsilon < \delta  } O_ \varepsilon  \} $.
Clearly,  
$(O_ \delta ) _{ \delta \in K } $
is still a cover of $X$.
Let $\alpha^*$ be the order type of $K$,
and let $f : \alpha ^* \to K$ be the order preserving
bijection.
Applying the assumption to the sequence
$(O _{ f( \gamma )}) _{ \gamma \in \alpha ^*} $, we get
$H \subseteq  \alpha^* $ with
order type $<\beta$, such that
$ \bigcup  _{ \gamma  \in H } O_ {f( \gamma )} = X $.
 This means that
$(O_ \delta ) _{ \delta \in f(H) } $
is  a cover of $X$
indexed by a set of order type $<\beta$. In particular,
it is a subcover of  $(O_ \delta ) _{ \delta \in \alpha } $
thus $[ \beta, \alpha ]$-\brfrt compactness is proved.
\end{proof}

\section{First examples} \labbel{examples} 

In this section we provide many examples showing that
ordinal compactness is not a ``trivial'' notion. In particular, 
it cannot be reduced to cardinal compactness. We also show 
that many of the results proved in Corollary \ref{cortransfer}  are the best possible ones, 
in the general case. On the contrary, 
we shall show in Section 
\ref{t1sec} that certain results can be improved
if we just assume that we are dealing with a $T_1$ topological space. 

In subsection \ref{basicex} we endow cardinals with several topologies, and characterize exactly the ordinal compactness properties they share.
Then in  Subsection \ref{disgsec} we give detailed results  about compactness properties
of disjoint unions, and show that taking disjoint unions
is a very flexible way to get more counterexamples.
Examples of a different kind shall be presented in Section \ref{exact}.

Finally, in Subsection \ref{rmkonsh} we discuss the technical notion
of a shifted sum of two ordinals, introduced in connection with compactness properties of disjoint unions.

\subsection{Basic examples} \labbel{basicex} 
\begin{definition} \labbel{toponcard}
We shall endow cardinals with several topologies.

As usual, the \emph{discrete} topology $\disc$ (on any set)
is the trivial topology in which every subset is open.

The \emph{initial interval topology} $\iit $ on some cardinal 
$ \lambda $ 
is the topology
whose open sets are the intervals of the form
$[0, \beta )$, with $ \beta  \leq  \lambda     $.

The \emph{order topology} $\ord$  on some cardinal 
$\lambda$ is the more usual topology; a base for this topology is given by
the intervals $( \alpha , \beta )$ ($ \alpha < \beta \leq \lambda $),
and $[0, \beta )$ ($  \beta \leq \lambda $).
 \end{definition}

\begin{examples}\labbel{exex} 
Let $\lambda$ be any cardinal, and  $\kappa$ be an infinite  regular cardinal.
\begin{enumerate}
\item 
 $(\lambda, \disc)$
is $[ \lambda ^+, \infty)$-\brfrt compact, and not 
$[ \alpha, \alpha ]$-\brfrt compact, for every nonzero $\alpha < \lambda ^+$.
\item
 $(\kappa, \iit )$ is not
$[ \kappa , \kappa  ]$-\brfrt compact, but it is
$[ \kappa  + 1, \infty)$-\brfrt compact, and 
$[2, \kappa )$-\brfrt compact.
\item 
If $\kappa> \omega $, then  
$(\kappa, \ord)$  
is a normal topological space which is $[ \kappa +\omega , \infty  )$-\brfrt compact, $[ \omega ,  \kappa  )$-\brfrt compact, and not $ [\kappa +n, \kappa +n] $-\brfrt compact, for each $ n \in \omega$.
\end{enumerate}    
\end{examples}

\begin{proof}
(1) is trivial.

(2) The sequence $[0, \beta ) _{\beta  < \kappa   } $
itself proves $[ \kappa , \kappa  ]$-incompactness,
since $\kappa$ is an infinite regular cardinal.

On the other hand, let $(O_ \delta ) _{ \delta \in \alpha}$ 
be a cover of  $(\kappa, \iit )$. If $O_ \delta = \kappa $, for some $ \delta \in \alpha$,
then clearly $\{ O_ \delta \}$ itself is a one-element subcover.

Suppose otherwise. Since $\kappa$ is regular, then necessarily
$\alpha  \geq \kappa $, and our aim is to extract a subcover of order type $\leq \kappa$. In fact, the subcover will turn out to be of order type  exactly $\kappa$.

By Lemma \ref{irredundant}, the result follows from the particular case
in which the cover $(O_ \delta ) _{ \delta \in \alpha}$ has the additional property that,
for every $\delta <  \alpha $, 
$O_ \delta $ is not contained in 
 $ \bigcup  _{ \varepsilon < \delta  } O_ \varepsilon  $. Suppose that the above condition is satisfied. Since each
$O_ \delta $ has the form $[0, \beta_ \delta )$, for some $\beta_ \delta < \kappa $,
then, by the above condition, $\beta_ \delta  < \beta _{ \delta' } $, 
for all pairs $\delta < \delta' < \alpha $.    

Since $(O_ \delta ) _{ \delta \in \alpha}$ 
is a cover of $ \kappa $, then 
$\sup _{\delta <  \alpha }  \beta_ \delta = \kappa $.
Thus, the sequence $ (\beta _ \delta ) _{ \delta < \alpha } $ 
is strictly increasing, and cofinal in $\kappa$, hence has order type
$\kappa$, since $\kappa$ is a regular cardinal.

(3) Let $(O_ \delta ) _{ \delta \in \alpha}$ 
be a cover of $(\kappa, \ord)$.

First, consider the case when some 
$O_ {\bar{\delta}} $ contains an interval of the form 
$(\varepsilon, \kappa )$,
for some $\varepsilon < \kappa $. 
Since
$[0, \varepsilon ]$ is compact, it is covered by
a finite number of the $O_ \delta $'s. If we add
$O_ {\bar{\delta}} $ to these, we get a finite subcover of 
$\kappa$, since
$ \kappa  = [0, \varepsilon ] \cup (\varepsilon, \kappa )$,
hence the conclusion holds in this case.

So we can suppose that
no $O_ \delta $ contains an interval of the form 
$(\varepsilon, \kappa )$, thus necessarily
$\alpha \geq \kappa $,
since $\kappa$ is regular. 
Since $(O_ \delta ) _{ \delta \in \alpha}$ 
is a cover, and each  $O_ \delta $
is a union of intervals, we have that, for every $\beta \in \kappa $,
with $\beta \not= 0$, there is an interval  
$I_ \beta =( \varepsilon _ \beta , \phi _ \beta )$, with
$ \varepsilon _ \beta < \phi _ \beta  < \kappa $, 
such that  $\beta \in I_ \beta $,
and  $I_\beta \subseteq O _{ \delta ( \beta )} $,
for some $\delta( \beta ) \in \alpha $.  
For every nonzero
$\beta \in \kappa $,
choose some $I_ \beta $
and some $\delta( \beta ) \in \alpha $ as above.  
The function
 $f : \kappa \setminus \{ 0\} \to \kappa  $ defined
by $f( \beta  )= \varepsilon _ \beta  $ is regressive, hence constant on a 
set $S$ stationary in $\kappa$,
say,  $f( \beta  )= \bar{ \varepsilon }  $, for $ \beta  \in S$.

Let $D = \{ \delta \in \alpha  \mid   \delta = \delta ( \beta ), 
\text{ for some } \beta \in S\}  $.
For $ \delta \in D$, let 
$\eta_ \delta = \sup \{  \eta < \kappa \mid O_ \delta \supseteq (\bar{ \varepsilon }, \eta) \} $,
and let $J_ \delta = (\bar{ \varepsilon }, \eta _ \delta )$. 
  Thus, 
$O_ \delta \supseteq J_ \delta $, for $ \delta \in D$. 
Moreover, 
$ I_\beta \subseteq J _{ \delta ( \beta )}$, for $\beta \in S$,
since, if $\beta \in S$, then
$( \bar{ \varepsilon } , \phi _ \beta ) =I_ \beta \subseteq O _{ \delta ( \beta )} $.
We now show  that $(J _ \delta ) _{ \delta \in D} $ 
is a cover of 
$(\bar{ \varepsilon }, \kappa )$. Indeed, 
since $S$ is stationary, in particular, cofinal, 
then, for every $\beta'$ with
$\bar{ \varepsilon } < \beta ' < \kappa $,
there is $\beta > \beta '$,
such that $\beta \in S$, thus 
$\beta' \in I _ \beta $, since
$\bar{ \varepsilon } < \beta ' < \beta \in I_ \beta =( \varepsilon _ \beta , \phi _ \beta ) $,
hence 
$\beta' \in J _{ \delta ( \beta )} \supseteq I_\beta $.
 
Since 
$(\bar{ \varepsilon }, \kappa )$ is
order-isomorphic to $\kappa$, and, through
this isomorphism, the $J_ \delta $'s 
correspond to open sets in the $\iit$ topology,
we can apply (2) 
in order to  get a subset $E \subseteq D \subseteq \alpha $ 
such that $E$ has order type $ \leq \kappa $, 
and 
$(J _ \delta ) _{ \delta \in E} $
 covers
$(\bar{ \varepsilon }, \kappa )$.
Hence also 
$(O _ \delta ) _{ \delta \in E} $
 covers
$(\bar{ \varepsilon }, \kappa )$.

Since $ \kappa = [0, \bar{ \varepsilon }] \cup (\bar{ \varepsilon }, \kappa )$,
and  $[0, \bar{ \varepsilon }]$ is compact, it is enough to add
to $E$ a finite number of elements from 
the original sequence $(O_ \delta ) _{ \delta \in \alpha}$,
in order to get a cover of the whole $\kappa$.
Since we have added a finite number of elements
to a sequence of order type $ \leq\kappa$, we get a 
cover of $\kappa$ which has order type 
$< \kappa + \omega $, and which is
 a subsequence of the original sequence. Thus, we have proved 
$[\kappa + \omega, \alpha ]$-compactness. 

In order to finish the proof, we have to show that,
for each $ n \in \omega$,  $(\kappa, \ord)$  
is   not $ [\kappa +n, \kappa +n] $-\brfrt compact.
An easy counterexample is given by the sequence
$(O_ \delta ) _{ \delta \in \kappa +n } $ 
defined by 
\[ 
O_ \delta  =\begin{cases}
[n, n+\delta )&    \text{if  $ \delta < \kappa   $}, \\
\{ m \}  &    \text{if  $ \delta = \kappa +m $, with $m <n$}  \qedhere
\end{cases} 
\] 
\end{proof}

The situation appears in a clearer light if we introduce
an ordinal variant of the  Lindel\"of number of a space.

\begin{definition} \labbel{lindel}
The  \emph{Lindel\"of ordinal} of $(X, \tau )$
is the smallest ordinal $\alpha$ such that  
$(X, \tau )$ is $ [ \alpha , \infty) $-\brfrt compact.

Compare the above definition with the classical notion
of the \emph{Lindel\"of number} of a topological space $X$,
which is the 
smallest \emph{cardinal} $ \mu $ such that  
$X$ is $ [ \mu ^+, \infty) $-\brfrt compact
(the \emph{Lindel\"of number} is a distinct notion from the \emph{Lindel\"of\/$^+$ cardinal}
defined in the introduction.)

 \end{definition}   

Thus, the 
Lindel\"of number $\mu$ of $X$ is determined by its
Lindel\"of ordinal $\alpha$. Indeed, $\mu $
is the predecessor of 
$ \alpha $,
if
$\alpha$ is a successor cardinal, and
$\mu= | \alpha |$ otherwise.
On the other hand, in general, the 
 Lindel\"of ordinal cannot be determined by
the Lindel\"of numbers, as shown by Example \ref{exex}.
Indeed, taking $\lambda= \kappa $ regular and uncountable,  all the spaces in Examples \ref{exex} 
have  Lindel\"of number equal to $\kappa$, however,
their Lindel\"of ordinals are, respectively, $\kappa^+$,
$\kappa+1$, and $\kappa+ \omega $.    
Other possibilities for the Lindel\"of ordinal
are presented in Examples
\ref{exgen}, \ref{samecard} and \ref{condiscr}.
On the other hand, restrictions on the possible
values Lindel\"of ordinals can assume are given
in Corollary \ref{lindw} for spaces of small cardinality, 
and in Corollary \ref{lindt1} for $T_1$ spaces.

\begin{remark} \labbel{w}   
Examples \ref{exex} also show that ordinal compactness cannot be
determined exclusively by the cardinal compactness properties 
enjoyed by some space.
 For example, $X_1 = (\omega, \iit)$
is $[ \omega  + 1, \infty)$-\brfrt compact, hence, by Proposition \ref{simple}(1), it
is $[ \alpha , \alpha ]$-\brfrt compact, for every ordinal $\alpha > \omega $.
On the other hand, $ X_2=(\omega, \disc)$ 
is $[ \omega _1, \infty)$-\brfrt compact, but not 
$[ \alpha , \alpha ]$-\brfrt compact, for every countable ordinal  $\alpha$.
Thus, $X_1$ and $X_2$ are $[ \lambda , \mu]$-\brfrt compact exactly for the
same pairs of infinite cardinals $\lambda$ and $\mu$,
but there are many ordinals $\alpha$ for which  
$X_1$ is $[ \alpha , \alpha ]$-\brfrt compact, but $X_2$ is not.

Example \ref{samecard} below furnishes two normal topological spaces which are $[ \lambda , \mu]$-\brfrt compact exactly for the
same pairs of cardinals $\lambda$ and $\mu$, no matter whether finite or infinite,
but not $[ \alpha  , \alpha ]$-\brfrt compact for the same ordinals.
 \end{remark}

\subsection{Disjoint unions} \labbel{disgsec} 
In order to refine Examples \ref{exex}, we need some definitions.

\begin{definition} \labbel{disgdef}   
If $X_1$ and $X_2$ are sets, with $\tau_1$, $\tau_2$
respective families of subsets,
the \emph{disjoint union}  $  (X_1 \cupdot X_2, \tau) $  of  
$(X_1, \tau _1)$ and $(X_2, \tau_2)$
is a set $  X_1 \cupdot X_2$ obtained by taking the union
of disjointed copies of $X_1$ and $X_2$, 
with $\tau$  being the family of all subsets of $ X_1 \cupdot X_2$ 
which either belong to (the copy of) $\tau_1$, or belong to $\tau_2$,
or are the union of a set in $\tau_1$ and a set in $\tau_2$.
Of course, in the case when $X_1$ and $X_2$  are topological spaces,
we get back the usual notion of disjoint union in the topological sense.
\end{definition} 

\begin{definition} \labbel{defshift}  
If $\alpha$ and $\beta$ are ordinals, we say that some ordinal 
$\gamma$ is a \emph{shifted sum} of $\alpha$ and $\beta$ if and only if 
$\gamma = I \cup J$, for  some (not necessarily disjoint) subsets 
$I , J \subseteq \gamma $
such that
 $I $ has order type $\alpha$ and 
 $I $ has order type $ \beta $.

Trivially, both $\alpha + \beta $ and $\beta+ \alpha $
are   shifted sums of $\alpha$ and $\beta$.
The (Hessenberg) natural sum $\alpha \oplus \beta $ 
is the largest possible shifted sum of $\alpha$ and $\beta$.
This is immediate from \cite[Theorem 1, I, II]{Car}, 
where the Hessenberg natural sum is denoted by 
$ \sigma(\alpha , \beta) $, and follows also from Proposition \ref{chsh} below.

However, there are other possibilities for shifted sums.
For example, $ \omega_1 + \omega $ is a   
  shifted sum of $ \omega_1$ and $ \omega+ \omega $.
A quite involved formula for determining all the possible 
shifted sums of $\alpha$ and $\beta$ shall be obtained 
in Proposition \ref{chsh}, by expressing ordinals in additive normal form.
The complication arises from the fact that, say,
though both
$ \omega^3 + \omega + 1$ 
and 
$ \omega^3 + \omega^2 +1$ 
are shifted sums of
$ \alpha = \omega^3 + \omega $ and
$ \beta =  \omega^2 + 1 $, on the contrary
$ \omega^3 + 1$
is not a shifted sum 
of $\alpha$ and $\beta$.
 
 If $ \alpha  $ and $  \beta$ are ordinals, we denote by
 $ \alpha  +^* \beta  $ the smallest ordinal $ \delta $ 
 larger than all the shifted sums of
$\alpha'$ and $\beta'$, for 
$\alpha' < \alpha $ and
$\beta' < \beta $.
Alternatively,
 $ \alpha  +^* \beta $ can be defined as 
$ \sup _{\alpha ' < \alpha ,  \beta ' < \beta   }    \alpha' \oplus \beta' +1    $.
 \end{definition}

We shall also need the following lemma.

\begin{lemma} \labbel{lemshifted}
Suppose that $\gamma $ is a shifted sum of $\alpha$ and $\beta$, 
that is, $\gamma = I\cup J$, with 
 $I$ having order type $\alpha$ and
 $J $ having order type $ \beta $.

Then the following additional property
is satisfied.
Whenever  
 $I^* \subseteq I$ has still order type $\alpha$, and 
 $J^* \subseteq J $  has still order type $ \beta $, 
then 
$I^* \cup J^*$ has still order type $\gamma$.
 \end{lemma}

\begin{proof}
Express $ \gamma $ in
additive normal form as
\[
 \gamma =
 \omega ^ {\eta_h}  + \omega ^ {\eta _{h-1}}   +
\dots
+ \omega ^ {\eta_1}  + \omega ^ {\eta_0} , 
\]
  for some  integer 
$h \geq 0$, and ordinals
$ \eta _h \geq \eta _{h-1} \geq \dots \geq \eta_1 \geq \eta_0  $.

Put $\gamma _{h+1}=0 $ and,
 for $i=0, \dots, h$, put
\[ 
\gamma_i= \omega ^ {\eta_h}  + \omega ^ {\eta _{h-1}}   +
\dots
+ \omega ^ {\eta _{i+1} }  + \omega ^ {\eta_i}. 
 \]
Consider the intervals
$K_i= [ \gamma _{i+1} , \gamma _i)$,
 for $i=0, \dots, h$. 
Clearly, 
each
  $K_i$ has order type $\omega ^ {\eta_i}$. Moreover, 
$\gamma$ is the disjoint union of the $K_i$'s. 

Fix some ${\bar \imath}$.
Since $\gamma= I \cup J$, then
$K_{\bar \imath} = (I \cap K_{\bar \imath}) \cup (J \cap K_{\bar \imath})$.
Since 
  $K_{\bar \imath}$ has order type $\omega ^ {\eta_{\bar \imath}}$,
then, by an easy property of such exponents, either 
$I \cap K_{\bar \imath}$ or
$J \cap K_{\bar \imath}$
has order type $\omega ^ {\eta_{\bar \imath}}$
(this is similar to, e. g., Hilfssatz  1 in \cite{lauchli}).
Suppose that, say,
$I_{\bar \imath}=I \cap K_{\bar \imath}$ 
has order type $\omega ^ {\eta_{\bar \imath}}$.
Let 
$I _\bullet=I \cap (K_h \cup \dots \cup K _{\bar \imath +1}) $,
and 
$I_\starr=I \cap (K _{\bar \imath-1}  \cup \dots \cup K _{0}) $,
and let $ \alpha _\bullet$, $ \alpha _\starr$ be their respective order types.
Since $\gamma$ is the union of the 
$K_{i}$'s, then
$I= I_\bullet\cup I_{\bar \imath} \cup I_\starr$.  
 Because of the relative way the elements of 
the  $K_{i}$'s are ordered in $\gamma$, we have that
$\alpha= \alpha _\bullet+ \omega ^ {\eta_{\bar \imath}}+ \alpha _\starr$. 
Notice that $\alpha _\starr < \omega ^ {\eta_{\bar  \imath}+1}$,
since the order type of $K _{i-1}  \cup \dots \cup  K _{0}$ is
 $ \omega ^ {\eta _{i-1}}   +
\dots
 + \omega ^ {\eta_0}< \omega ^ {\eta_{\bar \imath}} \cdot \omega = \omega ^ {\eta_{\bar \imath}+1}$.
Since $I^* \subseteq I$, 
then the order types of, respectively, 
$I^* \cap  I_\bullet $, $ I^* \cap  I_{\bar \imath} $, and $ I^* \cap I_\starr$
are $\leq$ than, respectively,
 $\alpha _\bullet $, $  \omega ^ {\eta_{\bar \imath}} $,
and $  \alpha _\starr$. However, since
both $I^* $ and $ I$ have order type $\alpha$, 
then necessarily
 $I^* \cap I _{\bar \imath}= I^* \cap K_{\bar \imath}$ has order type $\omega ^ {\eta_{\bar \imath}}$, since otherwise the order type of
$I^*$ would be strictly smaller than $\alpha= \alpha _\bullet+ \omega ^ {\eta_{\bar \imath}}+ \alpha _\starr$, since, as we mentioned, 
  $\alpha _\starr < \omega ^ {\eta_{\bar  \imath}+1}$. 

In a similar way, if $J \cap K_{\bar \imath}$
has order type $\omega ^ {\eta_{\bar \imath}}$, then
also $J^* \cap K_{\bar \imath}$
has order type $\omega ^ {\eta_{\bar \imath}}$.
Since the above argument works for each $i$,
we get that, for each $i=0, \dots, h$,
either   $I^* \cap K_i$ or $J^* \cap K_i$
contribute to $K_i$ with order type $\omega ^ {\eta_i}$,
that is, $(I^* \cup J^*) \cap K_i$ has order type $\omega ^ {\eta_i}$.
This, together with the definition of the $K_i$'s, implies that
 $I^* \cup J^*$ has order type $\gamma$.
 \end{proof}

In the next lemma 
we characterize the compactness properties of disjoint unions.
The lemma has not the most general form possible, but it  is quite good
for our purposes. 

\begin{lemma} \labbel{disg}
Assume the notation in Definitions \ref{disgdef} and \ref{defshift}.
\begin{enumerate} 
\item
Suppose that  $X_1$ is not $[ \alpha , \alpha ]$-\brfrt compact, and 
 $X_2$ is not $[ \beta , \beta  ]$-\brfrt compact.

If $\gamma$ is a shifted sum of $\alpha$ and $\beta$,
then $ X_1 \cupdot X_2$
is  
not $[ \gamma  , \gamma  ]$-\brfrt compact.

In particular,  $ X_1 \cupdot X_2$
is neither $[ \alpha + \beta , \alpha + \beta ]$-\brfrt compact,
nor $[ \beta +\alpha , \beta +\alpha ]$-\brfrt compact,
nor $[ \alpha \oplus \beta , \alpha \oplus \beta ]$-\brfrt compact.
\item
If $X_1$ is  $[ \beta _1 , \alpha ]$-\brfrt compact, 
and $X_2$ is $[ \beta _2 , \alpha  ]$-\brfrt compact,
then $ X_1 \cupdot X_2$
is $[ \beta _1 +^* \beta _2  , \alpha  ]$-\brfrt compact. 
\end{enumerate}  
\end{lemma}

\begin{proof} 
(1) Represent  $ \gamma $ as $  I \cup J$ as in the definition 
of a shifted sum, 
with $I $ of order type $\alpha$ and 
 $J$ of order type $ \beta $,
and let 
$f_1: I \to \alpha $ and 
$f_2: J \to \beta  $
be the order preserving bijections.

Let
$(O_ \delta ) _{ \delta \in \alpha } $ be a cover of $X_1$
witnessing
$[ \alpha , \alpha ]$-\brfrt incompactness,
and let 
$(P_ \varepsilon  ) _{ \varepsilon \in \beta  } $ be a cover of $X_2$
witnessing
$[ \beta , \beta ]$-\brfrt incompactness.

For $ \phi \in \gamma $, 
let $Q _ \phi \subseteq X_1 \cupdot X_2$ be defined by 

\[ 
Q _ \phi=\begin{cases}
O _{ \delta } &    \text{if  $ \phi \in I \setminus J$ and $\delta= f_1( \phi)$ },\\
P _{ \varepsilon  } &    \text{if  $ \phi \in J \setminus I$ and $ \varepsilon = f_2( \phi)$ },\\
O _{ \delta } \cup P _{ \varepsilon  } &    \text{if  $ \phi \in I \cup J$, $\delta= f_1( \phi)$, 
and $ \varepsilon = f_2( \phi)$}.\\
\end{cases}
\] 

By the definition of disjoint union, 
$(Q _ \phi ) _{\phi \in \gamma } $ is a cover 
of $ X_1 \cupdot X_2$ with elements in $\tau$.
Suppose that $H \subseteq \gamma $,
and that  $(Q _ \phi ) _{\phi \in H } $ is still a cover 
of $ X_1 \cupdot X_2$. Then it is easy to see that
$(O_ \delta ) _{ \delta \in f_1(H \cap I)} $ is a cover of $X_1$.
Since $(O_ \delta ) _{ \delta \in \alpha } $ witnesses the
$[ \alpha , \alpha ]$-\brfrt incompactness of $X_1$,
then $f_1(H \cap I)$ has order type $\alpha$,
hence also  $I^*=H \cap I$ has order type $\alpha$,
since $f_1$ is an order preserving  bijection. Similarly,
 $J^*=H \cap J$ has order type  $\beta$.

By Lemma \ref{lemshifted}, 
$H= H \cap \gamma = H \cap (I \cup J) = 
(H \cap I) \cup (H \cap J) =I^* \cup J^*$ has order type $\gamma$.  
Thus, $(Q _ \phi ) _{\phi \in \gamma }$ is a counterexample to
the $[ \gamma  , \gamma  ]$-\brfrt compactness
of $ X_1 \cupdot X_2$.

The last statement in (1) follows from  the remarks in Definition \ref{disgdef}.

(2)
Let
$(O_ \delta ) _{ \delta \in \alpha } $ be a cover of 
 $ X_1 \cupdot X_2$.

Let $I$ be the set 
of all $\delta \in \alpha $
such that either $O_ \delta = P_ \delta $,
for some $P_ \delta \in \tau _1$,
or
 $O_ \delta = P_ \delta  \cup Q_ \delta $,
for some (unique pair) $P_ \delta \in \tau _1$
and $Q_ \delta \in \tau _2$.
Similarly, let 
 $J$ be the set 
of all $\delta \in \alpha $
such that either $O_ \delta = Q_ \delta $,
for some $Q_ \delta \in \tau _2$,
or
 $O_ \delta = P_ \delta  \cup Q_ \delta $,
for some $P_ \delta \in \tau _1$
and $Q_ \delta \in \tau _2$.
Notice that
$I \cup J = \alpha $,
because of the definition of disjoint union.

Since $(O_ \delta ) _{ \delta \in \alpha } $ is a cover of 
 $ X_1 \cupdot X_2$, then  $(P_ \delta ) _{ \delta \in I} $ 
is a cover of $X_1$, and, since $I$
has order type $\leq \alpha$,
then, by  Remark \ref{ordset}, 
Proposition \ref{simple}(1), and   
the  $[ \beta _1 , \alpha ]$-\brfrt compactness of 
$X_1$, 
there is 
$I^* \subseteq I$
such that 
$I^* $
has order type $\beta'_1< \beta_1$,
and 
$(P_ \delta ) _{ \delta \in I^*} $ 
is still a cover of $X_1$.
Similarly,
there is 
$J^* \subseteq J$
such that 
$J^* $
has order type $\beta'_2< \beta_2$,
and 
$(Q_ \delta ) _{ \delta \in J^*} $ 
is  a cover of $X_2$.

Let  $\gamma$ 
be the order type of $I^* \cup J^* $.
Then $\gamma$ is a shifted sum of 
$\beta'_1$ and  $\beta'_2$,
thus $\gamma < \beta _1 +^* \beta _2$.
Since 
$(O_ \delta ) _{ \delta \in I^* \cup J^* } $ 
turns out 
to 
be a cover of 
 $ X_1 \cupdot X_2$,
the conclusion follows.
\end{proof} 

\begin{corollary} \labbel{lindofdisg}
Suppose that the Lindel\"of ordinal of $ X_1 $ is $\beta_1$, and that the Lindel\"of ordinal of $  X_2$
is $\beta_2$. Then the Lindel\"of ordinal of $ X_1 \cupdot X_2$
is   $\beta _1 +^* \beta _2 $.
\end{corollary} 

\begin{proof} 
The Lindel\"of ordinal of $ X_1 \cupdot X_2$
is   $ \leq \beta _1 +^* \beta _2 $, as an  immediate consequence of 
Lemma \ref{disg}(2).

 Hence, to prove equality, and in view of Proposition \ref{simple}(2), we have to show that, for every 
$\gamma <  \beta _1 +^* \beta _2 $, there is $\gamma''$ 
with $\gamma \leq \gamma ''< \beta _1 +^* \beta _2 $ 
and such that $ X_1 \cupdot X_2$ is not 
$[ \gamma''  , \gamma''  ]$-\brfrt compact.
Let $\gamma <  \beta _1 +^* \beta _2 $.
By the definition of $\beta _1 +^* \beta _2 $,
there are 
$ \beta _1' < \beta _1 $, $  \beta _2' < \beta _2 $,
and $\gamma' <  \beta _1 +^* \beta _2 $
such that 
$\gamma \leq \gamma '$ and 
$\gamma'$ is a shifted sum of
$ \beta _1'$ and $  \beta _2'  $.
By assumption, $X_1$ is not 
 $[ \beta _1'  , \infty )$-\brfrt compact, 
hence, by Proposition \ref{simple}(2), 
there is $\beta''_1 \geq \beta' _1$ such that
$X_1$ is not $[ \beta''_1 , \beta''_1  ]$-\brfrt compact,
and necessarily $\beta''_1 < \beta _1$.
Similarly, there is 
$\beta''_2 $ such that
$X_2$ is not $[ \beta''_2 , \beta''_2  ]$-\brfrt compact,
and  $ \beta '_2  \leq \beta''_2 < \beta _2$.

It follows trivially form the definition of a shifted sum,
and from $\beta'_1 \leq \beta'' _1$ and 
$\beta'_2 \leq \beta'' _2$, that there is some
shifted sum $\gamma''$ of $\beta''_1$
and $\beta''_2$ such that $\gamma' \leq \gamma ''$.   
By Lemma \ref{disg}(1), 
$ X_1 \cupdot X_2$ is not 
$[ \gamma '', \gamma '' ]$-\brfrt compact.
Since 
$\beta''_1 < \beta _1$ and
$\beta''_2 < \beta _2$, then 
$ \gamma ''< \beta _1 +^* \beta _2 $. Thus  $\gamma''$  is an ordinal as wanted.
\end{proof}  

We are now ready to present many improvements of
Examples \ref{exex}.

\begin{examples} \labbel{exgen}
 Let $\kappa$ be an infinite  regular cardinal, and $ n \in \omega$, $n \geq 2$.
\begin{enumerate} 
\item
If $X$ is the disjoint union of two copies
of $ \kappa $ with the initial interval topology 
$\iit $ 
of Definition \ref{toponcard},
then $X$ is not $[ \kappa ,  \kappa  ]$-\brfrt compact, 
not $[ \kappa+1 ,  \kappa +1 ]$-\brfrt compact, 
and
not $[ \kappa + \kappa , \kappa + \kappa  ]$-\brfrt compact, but it is
$[ \kappa  + \kappa + 1, \infty)$-\brfrt compact, 
$[ \kappa +2,  \kappa + \kappa  )$-\brfrt compact, and 
$[ 3, \kappa  )$-\brfrt compact.
Thus $X$ has Lindel\"of ordinal (Definition \ref{lindel}) $\kappa+ \kappa +1$. 
\item
More generally, if $X$ is the disjoint union of $n$  copies
of $ \kappa $ with the initial interval topology,
then $X$ is not $[ \kappa ,  \kappa  ]$-\brfrt compact, 
not $[ \kappa + \kappa , \kappa + \kappa   ]$-\brfrt compact, \dots,
not $[ \kappa \cdot n, \kappa \cdot n  ]$-\brfrt compact,
 but it is
$[ \kappa  \cdot n+ 1, \infty)$-\brfrt compact,
$[ \kappa +n,  \kappa + \kappa  )$-\brfrt compact,  
$[ \kappa + \kappa + n-1,  \kappa + \kappa + \kappa   )$-\brfrt compact, \dots, 
$[ \kappa \cdot (n-1) + 2, \kappa \cdot n  )$-\brfrt compact,
and 
$[ n+1,  \kappa  )$-\brfrt compact.
Its Lindel\"of  ordinal is $\kappa \cdot n +1$. 
\end{enumerate}    
\end{examples} 
 
\begin{example} \labbel{samecard} 
Suppose that $\kappa $ is regular and $ > \omega $, let 
$X_1= ( \kappa , \ord)$,  and 
let $X_2$ be the disjoint union
of two copies of $X_1$.

Then both $X_1$ and $X_2$
are $[\mu, \lambda ]$-compact, for every pair of 
infinite cardinals $\mu$ and $\lambda$ such that either
$  \kappa < \mu \leq \lambda $, or  $  \omega \leq \mu \leq \lambda < \kappa $;
furthermore, 
both $X_1$ and $X_2$
 are
not $[ \kappa, \kappa ]$-compact,
and not  $[ n, n ]$-compact,
for every positive integer $n$.
Thus, $X_1$ and  $X_2$ are  
 $[\mu, \lambda ]$-compact exactly for the same pairs of cardinals
$\mu$ and $\lambda$, whether finite or not.

However, $X_1$ is 
 $[ \kappa + \omega , \infty )$-compact,
while $X_2$ is not even 
 $[ \kappa + \kappa  , \kappa+ \kappa  ]$-compact.
Actually, $X_2$ is not  
 $[ \kappa + \kappa +n , \kappa+ \kappa +n ]$-compact,
for every $n< \omega$, but it
is  
$[ \kappa + \omega   , \kappa+ \kappa  )$-compact,
and
$[ \kappa + \kappa + \omega  , \infty  )$-compact.
Its Lindel\"of ordinal is $\kappa+ \kappa + \omega $. 
 \end{example}   

\begin{example} \labbel{condiscr} 
Suppose that $X_1$  is a nonempty set, and $\tau$ is a nonempty family of 
subsets of $X_1$. Suppose that $X_2$ is a discrete topological space of 
cardinality $\mu$, and that $X$ is the disjoint union of $X_1$ and $X_2$.  
 Then the following statements hold.
  \begin{enumerate} 
 \item 
If $X_1$ is not $[ \alpha   , \alpha ]$-compact,
$| \beta | \leq \mu$, and $\gamma$ is a shifted sum
of $\alpha$ and $\beta$, then $X$ is not 
 $[ \gamma  , \gamma  ]$-compact.
 \item 
If $X_1$ is  $[ \beta   , \alpha   ]$-compact,
then $X$ is
$[ \beta   +^* \mu^+ , \alpha  ]$-compact.
  \end{enumerate}

In particular, by adding a discrete finite set to Example \ref{exex}(2), we can get
a $[ \kappa +m+1, \infty)$-compact space which is not $[ \kappa +m, \kappa +m ]$-compact.
Thus we can have $\kappa+m+1$, as a Lindel\"of ordinal of some space.
In a similar way, by starting with Example \ref{exgen}, we can have
$\kappa \cdot n + m + 1$ as a Lindel\"of ordinal.   
\end{example}

\begin{proof}[Proofs]
Almost everything in Examples 
\ref{exgen}, \ref{samecard} and \ref{condiscr}  
follows from Proposition \ref{simple}, Examples \ref{exex} and Lemma \ref{disg}. 

An exception is  $[ \kappa +2,  \kappa + \kappa  )$-\brfrt compactness 
in Example \ref{exgen}(1), which is proved as follows.
Let $X$ be the disjoint union of two copies
of $ (\kappa , \iit )$, and consider an ordinal-indexed 
cover $\mathcal C$ of $X$. By  Example \ref{exex}(2),
there is a subsequence of  $\mathcal C$ which is a cover
of the first copy of $ (\kappa, \iit )$ and either has
order type $\kappa$, or consists of a single element, that is, has order 
type $1$. Similarly, there is a subsequence of  
$\mathcal C$ which is a cover
of the second copy and has the same possible order types.
By joining the above two partial subcovers,
we get a cover of the whole of $X$, whose order
type is a shifted sum of $ \beta _1$ and $\beta _ 2$, where 
the possible values  $ \beta _1$ and $\beta _ 2$ are either $\kappa$ or $1$.
Any such shifted sum, if $<\kappa + \kappa $,
must necessarily be $ \leq \kappa + 1$,
from which    
 $[ \kappa +2,  \alpha   ]$-\brfrt compactness
follows,
for every $\alpha $ with
$ \kappa + 2 \leq \alpha < \kappa + \kappa $.

The proofs of  $[ \kappa +n,  \kappa + \kappa )$-\brfrt compactness, 
$[ \kappa + \kappa + n-1,  \kappa + \kappa + \kappa   )$-\brfrt compactness, 
\dots\  
in \ref{exgen}(2), and of 
 $[ \kappa + \omega ,  \kappa + \kappa )$-\brfrt compactness
in \ref{samecard} are similar. 
\end{proof}

Many other similar examples can be obtained by combining
in various ways the examples in \ref{exex} with
Lemma \ref{disg}. Further counterexamples can be obtained
by applying disjoint unions to the examples we shall 
introduce  in Definition \ref{saa}.

\begin{example} \labbel{samecard2}
It is trivial to show that, for $\mu \leq \lambda $ infinite cardinals,
the disjoint union of two topological spaces is
$[ \mu, \lambda ]$-\brfrt compact 
if and only if the two spaces are both
$[ \mu, \lambda ]$-\brfrt compact
(this also follows from Lemma 3.8).

The space constructed in Example \ref{samecard} 
shows that, for ordinals, the disjoint union of two 
$[ \beta , \alpha  ]$-\brfrt compact  spaces is not necessarily
$[ \beta , \alpha  ]$-\brfrt compact. Just take $\alpha= \beta = \kappa + \kappa $,
for some regular $\kappa> \omega $, 
and consider the union of two disjoint copies of  $ ( \kappa , \ord)$.
 \end{example}

 One can also deal with 
the obviously defined notion of the disjoint union of an infinite family.
It appears to be promising also the possibility of considering a partial compactification
of 
an infinite disjoint union. This can be accomplished as follows.

 \begin{definition}\labbel{fr}
Suppose that $(X_i) _{i \in I} $
is a family of nonempty sets and, for each
$i \in I$, $\tau_i$ is a nonempty family of subsets of 
$X_i$. Suppose, for sake of simplicity, that 
each $\tau_i$ contains the empty set.

The
\emph{Frechet disjoint union}
$(X, \tau )$
of $(X_i, \tau _i) _{i \in I} $
is defined
as follows.

Set theoretically, $X = \{x\}  \cupdot \bigcupdot _{i \in I} X_i $ is the union of
(disjoint copies) of the $X_i$'s, plus a new element
$x$ which belongs to no $X_i$.

The members of $\tau$ are those subsets $O$ of $X$ 
which have one of the following two forms.
\[ 
O = \bigcupdot _{i \in I} O_i,
\] 
where $O_i \in \tau _i$, for every $i \in I$, or 
\[ 
O = \{x\}  \cupdot \bigcupdot _{i \in I} O_i,
\] 
where the $O_i$'s are such that, for some finite set $F \subseteq I$, it happens that
$O_i \in \tau _i$, for  $i \in F$, and
 $O_i = X_i$, for $i \in I \setminus F$.
\end{definition}

The above definition appears to be interesting, in the 
present context, since, as in Example \ref{samecard2},
$[ \beta , \alpha  ]$-\brfrt compactness of a Frechet disjoint union is not
necessarily preserved. However, (infinite) cardinal compactness and many
other topological properties are preserved, as asserted by the next proposition.

\begin{proposition}\labbel{disjfrech}
If $(X_i) _{i \in I} $ is a family of
topological spaces, then
their Frechet disjoint union
$X = \{x\}  \cupdot \bigcupdot _{i \in I} X_i $ is a topological space, 
and
is $T_0$, $T_1$, Hausdorff, regular, normal, $[\lambda,\mu]$-compact
(for given infinite cardinals $\lambda $ and $\mu$), 
has a base of clopen sets if and only if  so is (has) each $X_i$.
\end{proposition}

\begin{proof}
Straightforward. We shall comment only on regularity and normality.
For these, just observe that 
if $C$ is closed in $X$ and $C$
has nonempty intersection with infinitely many
$X_i$'s, then $x \in C$.
\end{proof}

Notice that the spaces in Examples \ref{exex}(2)  and  \ref{exgen} satisfy very 
few separation axioms. Indeed, just assuming that
$X$ is a $T_1$ topological space, it is impossible to construct
similar counterexamples. See Section \ref{t1sec}.

Curiously enough, Counterexample \ref{exgen}  cannot be generalized in a simple way
 in order to get  a space $X$ which is not $[ \kappa \cdot \kappa , \kappa \cdot \kappa   ]$-\brfrt compact, but which is, say,
$[ \kappa \cdot \kappa + \kappa , \kappa \cdot \kappa + \kappa   ]$-\brfrt compact.
 Such a counterexample exists (Remark \ref{w+ww2}), but we need a much more involved construction. Indeed, if $X$ is such a counterexample, then $|X|> \kappa $, 
as we shall show in the next section.

\subsection{A note on shifted sums and mixed sums} \labbel{rmkonsh} 
We now give the promised characterization
of those ordinals $\gamma$ which can be realized as a shifted sum
of two ordinals $\alpha$ and $\beta$.
 
Every ordinal $ \gamma $ 
can be expressed in a unique way
in
additive normal form as
\[
 \gamma =
 \omega ^ {\eta_h}  + \omega ^ {\eta _{h-1}}   +
\dots
+ \omega ^ {\eta_1}  + \omega ^ {\eta_0} , 
\]
  for some  integer 
$h \geq 0$, and ordinals
$ \eta _h \geq \eta _{h-1} \geq \dots \geq \eta_1 \geq \eta_0  $.
Hence to any ordinal $\gamma$ we can uniquely associate the 
finite string  $ \sigma ( \gamma ) $ of  ordinals 
in (not necessarily strictly) decreasing order $  \eta _h \eta _{h-1}\dots \eta_1 \eta_0  $. We are allowing the empty string, which is associated to the ordinal $0$. 

To every  string 
of ordinals $ \sigma = \eta _h \eta _{h-1}\dots \eta_1 \eta_0  $
we can associate the ordinal
$ \delta  ( \sigma ) =
 \omega ^ {\eta_h}  + \omega ^ {\eta _{h-1}}   +
\dots
+ \omega ^ {\eta_1}  + \omega ^ {\eta_0} 
$.
We are not necessarily assuming that the ordinals
in $ \sigma$ are in decreasing order. However,
an arbitrary string $ \sigma$ can  be \emph{reduced}
to a string $ \sigma^r$ whose elements 
are in (not necessarily strictly) decreasing order,
by taking out from $ \sigma$ all those elements 
which are followed from some strictly larger element.
Notice that, anyway,  
$ \delta  ( \sigma ^r) =  \delta  ( \sigma ) $,
since, for example, $ \omega^ \xi + \omega ^{ \xi'}= \omega ^{ \xi'}$,
if $\xi < \xi' $.  In particular, if
$\gamma= \delta ( \sigma )$,
then 
$ \sigma ( \gamma ) = \sigma ^r$,
since the correspondence between ordinals
and strings consisting of decreasing ordinals is bijective.  

We let $*$ denote string juxtaposition. 

\begin{proposition} \labbel{chsh}
Suppose that $\alpha$, $\beta$, $\gamma$ are ordinals,
and $ \sigma ( \gamma )= \eta _h \eta _{h-1}\dots \eta_1 \eta_0  $.
Then the following conditions are equivalent.
  \begin{enumerate}
    \item   
 $\gamma$ is a shifted sum of $\alpha$ and $\beta$.
      \item 
There are (possibly empty) strings
$ \sigma_h$, \dots, $ \sigma_0$
and    
$ \sigma'_h$, \dots, $ \sigma'_0$
such that 
  \begin{enumerate} 
   \item 
$  \alpha = \delta (\sigma_h * \sigma _{h-1} * \dots *  \sigma_0)  $, 
  \item 
$  \beta  = \delta (\sigma'_h * \sigma' _{h-1} * \dots *  \sigma'_0 ) $, 
\item   
for each $i=0, \dots, h$, either $ \sigma_i= \eta _i$,
or $ \sigma_i$ is empty, or every element of $ \sigma_i$
is $<\eta_i$,
\item   
for each $i=0, \dots, h$, either $ \sigma'_i= \eta _i$,
or $ \sigma'_i$ is empty, or every element of $ \sigma'_i$
is $<\eta_i$,   
  \item   
for each $i=0, \dots, h$, either $ \sigma_i= \eta _i$,
or   $ \sigma'_i= \eta _i$.
\end{enumerate}
     \item
Same as (2)
 with conditions (a) and (b) replaced by
  \begin{enumerate} 
   \item [(a$'$)]
$ \sigma( \alpha )= \sigma_h * \sigma _{h-1} * \dots *  \sigma_0  $, 
  \item [(b$'$)]
$ \sigma( \beta  )= \sigma'_h * \sigma' _{h-1} * \dots *  \sigma'_0  $. 
\end{enumerate}
 \end{enumerate} 
 \end{proposition}  

\begin{proof} 
For $i=0, \dots, h$,
define the intervals $K_{i}$ 
as in the proof of Lemma \ref{lemshifted}. Recall
that each $K_{i}$ has order type $ \omega ^{ \eta _i} $, 
that $\gamma$ is the disjoint union of the $K_{i}$'s, and that,
for every $i > i'$, each element of  $K_{i}$ precedes
every element of $K_{i'}$,
in the ordering induced by the ordering on $\gamma$.

(1) $\Rightarrow $  (2) By (1), $\gamma= I \cup J$, for some $I$ and $J$ of order types,
respectively, $\alpha$ and $\beta$.  For $i=0, \dots, h$, let 
$\alpha_i$ be the order type of $I \cap K_i$,
thus $\alpha= \alpha _h + \dots + \alpha _0$,
by the above properties of the $K_{i}$'s.
Put $ \sigma_i = \sigma ( \alpha _i)$.
Then (a) is satisfied, since 
$\delta (\sigma_h * \dots *  \sigma_0)
= \delta (\sigma_h) + \dots + \delta ( \sigma_0)$,
and since $\delta( \sigma ( \varepsilon ))= \varepsilon $,
for every ordinal $\varepsilon$.
Moreover, (c), too, holds, since the order type
of $\alpha_i$ is $\leq$  the order type of $K_{i}$, thay is,
 $ \omega ^{ \eta _i} $. Similarly, letting 
$ \beta _i$ be the order type of $J \cap K_i$,
and $ \sigma'_i = \sigma ( \beta  _i)$,
we have that (b) and (d) hold.
Finally, as remarked in the proof of Lemma \ref{lemshifted}, since 
$K_{i}= (I \cap K_i) \cup (J \cap K_i)$ has order type   
 $ \omega ^{ \eta _i} $,
then either $\alpha_i$ or $\beta_i$ has
order type   
 $ \omega ^{ \eta _i} $, thus (e) holds.

(2) $\Rightarrow $  (3)
Observe that $(\sigma_h * \dots *  \sigma_0 )^r=
\sigma^-_h * \dots *  \sigma^-_0 $,
for appropriate strings $ \sigma^-_i$,
such that each  $ \sigma^-_i$ is a substring of
$ \sigma_i$ (however, it is not necessarily the case that $ \sigma^-_i = \sigma^s_i$).
Then, by the last remark before the statement of the proposition,
$ \sigma( \alpha )= (\sigma_h * \dots *  \sigma_0 )^r = 
\sigma^-_h * \dots *  \sigma^-_0  $.
Moreover, if the $ \sigma_i$'s satisfy (c), then also the $ \sigma^-_i$'s
satisfy (c), since we are just taking out elements. Furthermore,
if $ \sigma_i= \eta _i$ and (c) holds, then this occurrence of
$ \eta _i$ is not deleted in $ \sigma_i^-$,
since $  \eta _i \geq \eta _{i'} $, for $i >i'$.
By taking further strings 
 $ \sigma'^{-}_i$ such that  $(\sigma'_h * \dots *  \sigma'_0 )^r=
\sigma'^{-}_h * \dots *  \sigma'^-_0 $,
and arguing as before, we get that the  $ \sigma^-_i$'s
and the $ \sigma'^-_i$'s witness (3).

(3) $\Rightarrow $  (2) is trivial, since $\delta( \sigma ( \varepsilon ))= \varepsilon $,
for every ordinal $\varepsilon$.

(2) $\Rightarrow $  (1)
For $i=0, \dots, h$, put $\alpha_i= \delta ( \sigma _i)$
and $ \beta _i= \delta ( \sigma' _i)$.
By Clauses (c)-(d), $\alpha_i $ and $  \beta _i $ are 
both $  \leq \omega ^{ \eta _i} $. 
Let $I_i$ be the initial segment of $K_i$ of order type
$\alpha_i$, and $J_i$ be the initial segment of $K_i$ of order type
$ \beta _i$. 
The definition is well posed, since the order type of 
$K_i$ is $ \omega ^{ \eta _i} $.
 By Clause (e), $I_i \cup J_i = K_i$.
If we put $I= I_0 \cup \dots \cup I_h$ and  
$J= J_0 \cup \dots \cup J_h$, then 
$I \cup J = K_0 \cup \dots \cup  K_h = \gamma $.
Notice that,
by the properties of the $K_i$'s, 
  $I$ has order type 
 $  \alpha _h + \dots + \alpha _0=
\delta (\sigma_h) + \dots + \delta ( \sigma_0)=
\delta (\sigma_h * \dots *  \sigma_0) = \alpha $,
by Clause (a). Similarly, by Clause (b),
$J$ has order type $\beta$, thus we are done.
\end{proof} 
Notice that, given $\alpha$ and $\beta$, there is only a finite number of 
ordinals $\gamma$ which are shifted sums
of $\alpha$ and $\beta$. Indeed, by Proposition \ref{chsh}, the elements
of $ \sigma( \gamma )$ are a (possibly proper) 
subset of the union of the sets of the elements
of $ \sigma( \alpha ) $ and of $ \sigma ( \beta )$
(counting multiplicities), and this can be accomplished 
only in a finite number of ways.

On the other hand, given $\gamma$, it might be the case  that $\gamma$ can be realized
 in infinitely many ways as a shifted sum. For example,
for every $ n < \omega$,
$ \omega^ \omega +1$   can be realized as the shifted sum
of $ \omega^ \omega $ and $  \omega ^n+1$.

Notice that $\gamma$ is the natural sum
$\alpha \oplus \beta $ of $\alpha$ and $\beta$ if and only if
a  representation as  in Proposition \ref{chsh} exists
in such a way that, 
for each $i=0, \dots, h$, either $ \sigma_i= \eta _i$ and $ \sigma'_i$ is empty,
or   $ \sigma'_i= \eta _i$ and $ \sigma_i$ is empty.

The notion of a shifted sum is related to a known  similar notion, 
usually called mixed sum (\/\emph{Mischsumme}\/, \cite{neumer,lauchli}).
In our notation, $\gamma$ is a \emph{mixed sum}
of $\alpha$ and $\beta$ if and only if and only $\gamma$ can be realized as a shifted sum of $\alpha$ and $\beta$ as in Definition  \ref{defshift},
with the additional assumption that $I \cap J = \emptyset $.

\begin{proposition} \labbel{mixed} 
Under the assumptions in Proposition \ref{chsh},
we have that $\gamma$ is a mixed sum of $\alpha$ and $\beta$ 
if and only if Condition (2) (equivalently, Condition (3)) in \ref{chsh} 
holds with the following additional clause
  \begin{enumerate}   
 \item [(f)]
For each $i=0, \dots, h$, if $\eta_i=0$, then either $ \sigma_i$
or   $ \sigma'_i$ is the empty string.
 \end{enumerate} 
\end{proposition} 
 
\begin{proof} 
If $\gamma$ is a mixed sum of $\alpha$ and $\beta$, then, in particular,
it is a shifted sum, hence the conditions in Proposition \ref{chsh}(2)(3) hold.
In order to prove (f), 
notice that, if $\gamma$ is a mixed sum of $\alpha$ and $\beta$, and 
$ \eta _i=0$, then $|K_i|=1$, hence 
either $I_i$ or $J_i$ is empty, since, 
in the present situation, they are disjoint and contained in 
$K_i$, thus (f) follows.

It remains to show how to get disjoint $I_i$ and $J_i$, for each $i$, in the proof of
(2) $\Rightarrow $  (1)  (hence we get disjoint $I$ and $J$, since the $K_i$'s are pairwise disjoint). If $\eta_i=0$, this follows from Clause (f).
Otherwise, observe that any set of order type $ \omega ^{ \eta _i} $
can always be expressed as the union of two disjoint subsets
having prescribed order types $\alpha_i$ and $\beta_i$, provided
that $\alpha_i$ and $\beta_i$ are both $ \leq  \omega ^{ \eta _i} $,
and their maximum is  $ \omega ^{ \eta _i} $.   
\end{proof}
 
A somewhat similar characterization of those ordinals $\gamma$ which
can be expressed as a mixed sum of $\alpha$ and $\beta$ has been given in 
\cite{lauchli}. Actually,  \cite{lauchli} deals with mixed sums with possibly
more than two summands. Also the results presented here can be easily generalized to the case of more than two summands. We leave this to the reader.
 
We now discuss in more details the relationship between the notions of a shifted sum and of a mixed sum. It turns out that the only difference is made by the ``finite  tail'' 
of $\gamma$, that is, if $ \gamma = \gamma ^\bullet +m$, with
$\gamma^\bullet$ limit, then the ways $\gamma^\bullet$ can be realized as a shifted sum determine the ways $\gamma$ can be realized as a mixed sum.  

\begin{corollary} \labbel{mixshif}
Let $\alpha$, $\beta$, and $\gamma$ be ordinals.
  \begin{enumerate}    
\item   
Suppose that  $  \gamma $ is a limit ordinal. Then $\gamma$ 
is a mixed sum of $\alpha$ and $\beta$ if and only if $\gamma$
is a shifted sum of $\alpha$ and $\beta$ (and, if this is the case, then 
either $\alpha$ or $\beta$ is limit, but not necessarily both).
\item  
More generally, suppose that $\gamma= \gamma^\bullet + m$, with $\gamma^\bullet$ limit, and $ \omega > m \geq 0$.
Then $\gamma$ 
is a mixed sum of $\alpha$ and $\beta$ if and only if 
there are integers $n, p \geq 0$ such that $n +p = m$,
$\alpha$ has the form $ \alpha^\bullet + n$,   
$\beta$ has the form $\beta^\bullet + p$, 
and
$\gamma^\bullet$
is a shifted sum of $\alpha^\bullet$ and $\beta^\bullet$
 (one of $\alpha^\bullet$ and $\beta^\bullet$ must thus be limit, but not necessarily both). 
\end{enumerate} 
 \end{corollary}

\begin{proof} 
(1) If $\gamma$ is limit, and
$ \sigma ( \gamma )= \eta _h \eta _{h-1}\dots \eta_1 \eta_0  $,
then all the $\eta_i$'s are $>0$, thus Clause (f) in Proposition \ref{mixed}
is automatically satisfied.

(2) If $\gamma= \gamma^\bullet + m$, then 
$\eta_i =0$ exactly for $i=0, \dots, m-1$,
thus  $ \sigma ( \gamma ^\bullet)= \eta _h \eta _{h-1}\dots \eta_m $. 
The conclusion now follows easily from (1) and Propositions \ref{chsh}
and \ref{mixed}. 
\end{proof}  

Notice that the notions of a shifted sum and of a mixed sum are distinct.
Indeed, it follows easily from Proposition \ref{chsh} that 
the smallest shifted sum of $\alpha$ and $\beta$ is $\sup \{ \alpha, \beta \} $.
However, the smallest mixed sum of, say, $ \omega+1$ and 
$ \omega+2$ is $ \omega+ 3 > \sup \{ \omega +1, \omega +2 \} $.    
In general, as a corollary of Proposition \ref{mixed}, we obtain 
a result by Neumer \cite{neumer}:
for $ \alpha = \alpha^\bullet +n $ and    $ \beta =\beta^\bullet+ p$,
where
$ \alpha^\bullet $ and    $\beta^\bullet$ are limit ordinals,
 the smallest 
mixed sum of   $ \alpha$ and    $\beta$
is $\alpha^\bullet + n + p$, if $ \alpha^\bullet =\beta^\bullet$,
and $ \sup \{ \alpha,\beta \}$,
if $ \alpha^\bullet \not=\beta^\bullet$.

\section{Some indispensability arguments and spaces of small cardinality} \labbel{small}

As we mentioned, a discrete space of cardinality $\lambda$ is
not $[ \alpha, \alpha ]$-compact, for every ordinal $\alpha$ of cardinality 
$ \leq \lambda$.
In a more general way, we can exhibit plenty of spaces which behave 
as discrete spaces, that is, for which 
ordinal (in-)compactness reduces to cardinal (in-)compactness.
This is the theme of the first propositions in the present section.
Then we proceed to prove a more sophisticated result, Theorem 
\ref{cork}, which implies that, if we restrict ourselves to spaces of
cardinality $\kappa$, then $[ \alpha, \alpha ]$-compactness is equivalent to
$[ \beta , \beta  ]$-compactness, for a large set of limit ordinals 
$\alpha$ and $\beta$ 
of cardinality $\kappa$. 
 In particular, for countable spaces, Corollary 
\ref{corw} shows that
$[ \alpha, \alpha ]$-compactness  becomes trivial above $ \omega \cdot \omega $. The above mentioned results  imply that  the relatively simple examples introduced in the previous section are really far from exhausting all possible kinds of counterexamples. Indeed, further and more involved counterexamples shall be constructed in the next section. In fact, in the next section we shall prove some equivalences which show that Proposition \ref{transfer}
cannot be improved.

In order to carry on the proof of the next proposition, we need a definition.

\begin{definition} \labbel{indisp}   
If $(O_ \delta ) _{ \delta \in \alpha } $ is a cover of $X$, let us say that 
some $O_ {\bar{\delta}} $ is \emph{indispensable} if and only if  
every subcover of $(O_ \delta ) _{ \delta \in \alpha } $ must contain
$O_ {\bar{\delta}} $. Equivalently,  $O_ {\bar{\delta}} $ is indispensable
if and only if there is $x \in O_ {\bar{\delta}} $ such that 
 $ x \not\in \bigcup  _{ \delta \in \alpha, \delta \not=\bar{\delta}  } O_ \delta  $.
 \end{definition}

For example, if $X$ is a topological space with the discrete topology, and 
$(O_ \delta ) _{ \delta \in \alpha } $ is a cover of $X$
consisting of (all) singletons, then each element of this cover is indispensable.

\begin{proposition} \labbel{ord=card}
Suppose that $\alpha$ is a nonzero ordinal, $\lambda$ is an infinite cardinal, 
and $(X, \tau )$ has some cover $(O_ \delta ) _{ \delta \in \alpha   } $
having at least $\lambda$ indispensable elements.
  \begin{enumerate}   
 \item  
If $| \alpha |= \lambda $, then $X$ is not
$[ \beta , \beta  ]$-\brfrt compact, for every ordinal $ \beta $ 
with $| \beta  |= \lambda $.
\item
If $\tau$ is closed under unions, then
$X$ is not
$[ \beta , \beta  ]$-\brfrt compact, for every nonzero ordinal $ \beta $ 
with $ |\beta | \leq \lambda $.
 \end{enumerate} 
 \end{proposition}  

\begin{proof} 
(1) Let $| \beta  |= \lambda $. Rearrange the sequence $(O_ \delta ) _{ \delta \in \alpha  } $
as $(O'_ \varepsilon  ) _{ \varepsilon  \in \beta  } $
in such a way that, in this latter sequence, the subsequence of the indispensable elements  has order type $ \beta $. This is always possible, since $\lambda$ 
is an infinite cardinal, $| \beta| = \lambda$, and there are $\lambda$-many indispensable elements in the original sequence. For example, 
if $\mu$ is the cardinality of the set of 
 non indispensable
elements (it may happen that $\mu=0$),
choose a subset $Z \subseteq \lambda $
with $| Z| = \mu$ and 
such that $| \lambda \setminus Z| = \lambda $,   
assign to non indispensable elements
only positions in  $ Z$,
and assign all the other positions in $\beta \setminus Z$
to all indispensable elements.

Every subcover of $(O'_ \varepsilon  ) _{ \varepsilon  \in \beta } $
must contain all of its indispensable elements, thus has order type $ \beta $.
This implies that $X$ is not $[ \beta , \beta  ]$-\brfrt compact.

(2) Let
$| \beta  | \leq \lambda $, say
$| \beta  |= \nu $.
Consider a new cover of $X$ obtained by choosing $\nu$-many
indispensable $O_ \delta$'s
and joining all the remaining $O_ \delta$'s into one of them (it is still in $\tau$,
since $\tau$ is closed under unions). If $\nu$ is finite,
then the result is trivial. Otherwise, it
 is obtained by applying  (1), with $\nu$ in place of $\lambda$,  to this new cover.
\end{proof} 

In Section \ref{t1sec}  we shall use arguments similar 
to those used in the proof of
Proposition \ref{ord=card}
 in order to prove results about
compactness properties of
$T_1$ spaces.

Theorem \ref{cork}  below is a far more sophisticated result
than Propositions \ref{ord=card}. Recall
that $+$ and  $\cdot$ denote, respectively,  ordinal sum and product.
Moreover, also exponentiation, if not otherwise specified,
will denote \emph{ordinal exponentiation}.

\begin{lemma} \labbel{nonoltreprod}
Suppose that $\kappa$ is an infinite regular cardinal
and $\alpha$ is an ordinal of the form $ \alpha_1+ \kappa^ \varepsilon $,
for some ordinals $\alpha_1 \geq 0$ and $\varepsilon >1$
such that $\varepsilon$ is either a successor ordinal, or
$\cf \varepsilon = \kappa $.
Suppose further that  $|X|= \kappa $, and that $(X, \tau )$  is not $[\alpha, \alpha]$-\brfrt compact.

Then $(X, \tau )$ is not $[ \alpha'  , \alpha' ]$-\brfrt compact, for every  limit ordinal
$\alpha'  $ of the form $\alpha'= \kappa \cdot \alpha' _1$,
for some $\alpha'_1 > 0$ with
$|\alpha'_1| \leq \kappa  $.

If, in addition, $\tau$ is closed under unions (in particular, if $\tau$ 
is a topology on $X$), then 
$(X, \tau )$ is not $[ \alpha'  , \alpha' ]$-\brfrt compact, for every  ordinal
$\alpha'  $ with  $| \alpha'| \leq \kappa $.
\end{lemma}  

\begin{proof} 
 Suppose that $(O_ \delta ) _{ \delta \in \alpha} $ is a counterexample 
to $[\alpha, \alpha]$-\brfrt compactness.
In particular, for every $\beta < \alpha$, we have
$ \bigcup _{ \delta < \beta } O_ \delta \subset X $ properly.

We shall show  a little more.
 
\begin{claim} 
  For every $\beta < \alpha$, there are  
$x \in X \setminus \bigcup _{ \delta < \beta } O_ \delta $
and $ \gamma  _x < \alpha$ such that  
 $x \not \in\bigcup _{ \delta \geq \gamma_x  } O_ \delta $
(hence, 
$x \in \bigcup _{ \beta \leq \delta <\gamma _x } O_ \delta $,
since $(O_ \delta ) _{ \delta \in \alpha} $  is a cover of $X$).
\end{claim}   
 
\begin{proof}[Proof of the Claim]
Suppose by contradiction that the statement in the claim  fails. Then, for some given  $\beta < \alpha$, we have that, 
for every $x \in X \setminus \bigcup _{ \delta < \beta } O_ \delta $,
there are arbitrarily large  indexes $\delta < \alpha$
such that $x \in O_ \delta $. Fix some $\beta$
as above, and enumerate the elements in
  $X \setminus \bigcup _{ \delta < \beta } O_ \delta$ as
$(x_ \gamma ) _{ \gamma \in \kappa '} $, with 
$\kappa' \leq \kappa $ (here we are using the assumption that $|X| \leq \kappa $). 
 
We shall define by transfinite induction a strictly increasing sequence 
$( \delta _ \gamma ) _{ \gamma \in \kappa '} $ such that 
$x_ \gamma  \in O _{ \delta _ \gamma } $, for every $ \gamma \in \kappa '$.
 First, choose some  $\delta_0 < \alpha$ 
 such that $x_0 \in O _{ \delta _0} $. 

Suppose that $\gamma < \kappa '$, and that 
 $( \delta _{\gamma '}) _{ \gamma' < \gamma } $
have already been defined. 
Notice that, by the assumption on $\varepsilon$,
the cofinality of 
$ \alpha = \alpha_1+ \kappa^ \varepsilon $
is $\kappa$.
Since $\gamma < \kappa ' \leq  \kappa $,
and  $\kappa$ is regular,
then 
 $ \sup _{ \gamma' < \gamma } \delta _{\gamma '} < \alpha $.
Hence,  by the
first paragraph in the proof, there is some $\delta_ \gamma > \sup _{ \gamma' < \gamma } \delta _{\gamma '}$ such that  $x_ \gamma  \in O _{ \delta _ \gamma } $.

Notice that 
$\{   \delta _ \gamma  \mid  \gamma \in \kappa ' \}$
has order type $ \kappa ' \leq \kappa$.  
Hence, if we put 
$D = [0, \beta ) \cup \{   \delta _ \gamma  \mid  \gamma \in \kappa ' \} $,
then $D$ has order type $ \leq \beta + \kappa '$.
Notice that $ \beta + \kappa '< \alpha $, since  $\alpha $
is of the form $  \alpha_1+ \kappa^ \varepsilon $
with $\varepsilon > 1$, hence each final subset   of $\alpha$  has
 order type $\kappa^ \varepsilon > \kappa $.

However, by construction,
$\bigcup _{ \delta \in  D } O_ \delta =X$,
hence we have found 
a subcover of 
$(O_ \delta ) _{ \delta \in \alpha} $
of order type $<\alpha$, and this contradicts
the assumption that 
$(O_ \delta ) _{ \delta \in \alpha} $ witnesses 
the failure of  $[\alpha, \alpha]$-\brfrt compactness
of $X$.

We have reached a contradiction, thus the claim is proved.
\qedhere$_ {Claim}$ 
\end{proof}

\emph{Proof of Lemma \ref{nonoltreprod} (continued)}
Now we are going to construct 
by transfinite induction two sequences 
$(x_ \xi) _{ \xi \in \alpha ''} $ 
and 
$( \gamma _ \xi) _{ \xi \in \alpha ''} $,
 for some ordinal 
$ \alpha '' \leq \alpha $, such that  
\begin{enumerate} 
\item
$x_ \xi$ belongs to $X$, for every $ \xi < \alpha ''$,
\item 
$ \gamma _{\xi'} < \gamma _ \xi<\alpha$, for every $ \xi' <\xi < \alpha ''$,
\item
$ \gamma  _0 =0$, $( \gamma _ \xi) _{ \xi \in \alpha ''} $
is continuous, and $\sup_{ \xi \in \alpha ''}  \gamma _ \xi = \alpha  $,
\item
$x_ \xi \in 
\bigcup \{ O_ \delta \mid  \gamma _ \xi \leq \delta < \gamma _{ \xi + 1}   \} $, for every $ \xi < \alpha ''$,
\item
$x_ \xi \not\in 
\bigcup \{ O_ \delta \mid  \delta \in [0, \gamma _ \xi) \cup  [\gamma _{ \xi + 1}  \alpha ) \}, $ for every $ \xi < \alpha ''$.
 \end{enumerate}   
 
Put $ \gamma  _0 =0$. 
By applying the claim to $\beta= \gamma _0=0$,
we  get $x_0 \in X $
and $ \gamma  _1 < \alpha$ such that  
$x_0  \in\bigcup _{ \delta <\gamma_1} O_ \delta $ and 
$x_0 \not \in\bigcup _{ \delta \geq \gamma_1} O_ \delta $.
 
Suppose that 
$x _{ \xi } $ 
and 
$ \gamma _{\xi+1}  $
have been already defined, for some 
$\xi$. Apply the claim
to $ \beta = \gamma _{\xi+1} $,
in order to obtain   
$x_{\xi+1} $
and $ \gamma  _{\xi+2} < \alpha $.

Now suppose that  $\xi$ is a limit ordinal, and that 
$x _{ \xi '} $ 
and 
$ \gamma _{\xi'}  $
have already been  defined, for all 
$\xi' < \xi$.
If  $\sup_{ \xi' < \xi }  \gamma _{\xi'}  = \alpha  $,
take $\alpha''= \xi$, and terminate the induction.
Otherwise, let 
$ \gamma _ \xi = \sup_{ \xi' < \xi }  \gamma _{\xi'} $.
Then
apply the claim
with 
$\beta =  \gamma _{\xi} $,
in order to obtain   
$x_\xi $
and $ \gamma  _{ \xi+1} $.

It is immediate to show that the sequences constructed in such a way
satisfy (1)-(5) above.

Notice that, since $|X|= \kappa $, and $X$  is not $[\alpha, \alpha]$-\brfrt compact, then
necessarily $|\alpha | \leq \kappa $.
On the other hand,  
$\alpha \geq \kappa $, since   $ \alpha = \alpha_1+ \kappa^ \varepsilon $,
for $\varepsilon >1$.   Hence $|\alpha | =\kappa $.
Moreover, by (2) and (3), and since  
$ \cf\alpha  = \kappa $,
we also get $ \cf\alpha''  = \kappa $, thus 
$|\alpha'' | =\kappa $, since $ \alpha'' \leq \alpha  $.

If we assume that $\tau$ 
is closed under unions, then the 
proof can be concluded in a rather simple way.
Indeed, by letting 
$U_ \xi =
\bigcup \{ O_ \delta \mid  \gamma _ \xi \leq \delta < \gamma _{ \xi + 1}   \} $,
for $\xi < \alpha ''$,
we have that 
$x_ \xi \in U _ \eta $
if and only if $\xi =\eta$.
Thus
$(U_ \xi) _{\xi < \alpha ''}$
is a cover,
by (3), and 
since 
$(O_ \delta ) _{ \delta \in \alpha} $ is a cover.
Moreover, $(U_ \xi) _{\xi < \alpha ''}$
consists of $|\alpha'' | =\kappa $
indispensable elements,
hence we are done
 by Proposition \ref{ord=card}(2).

It remains to prove the theorem
without the assumption that $\tau$ is closed 
under unions, and this involves some technical computations.
Hence, suppose that $\alpha'= \kappa \cdot \alpha' _1$,
for some $\alpha'_1 > 0$ with
$|\alpha'_1| \leq \kappa  $.

Partition    
$\alpha''$ into  
$ \alpha' _1$-many classes
$(Z_ \eta ) _{ \eta <\alpha' _1} $,
in such a way that   
$|Z_ \eta | = \kappa $, for every $  \eta <\alpha' _1 $.
This is possible, since
$| \alpha ''|= \kappa $,
and  
$ |\alpha' _1| \leq \kappa $.
For $  \eta < \alpha' _1 $,
put 
$I _ \eta = [ \kappa \cdot \eta , \kappa \cdot ( \eta +1) )$,
and
$W_ \eta= \bigcup _{ \xi \in Z_ \eta  } [ \gamma _ \xi, \gamma _{\xi +1} )$.
Notice that
$| W _ \eta |= \kappa $, for every $\eta <   \alpha' _1$.
For each $\eta$,  
 let $f_ \eta $ be a  bijection from 
$I _ \eta $ onto
 $W_ \eta$.
Notice that 
$\alpha' = \bigcup _{ \eta < \alpha' _1  } I _ \eta $,
and that each $I _ \eta $ has order type $\kappa$. 
Rearrange the original cover 
$(O_ \delta ) _{ \delta \in \alpha} $
as 
$(O'_ \zeta  ) _{ \zeta  \in \alpha'} $
according to the following rule.

If $\zeta \in \alpha' $, then
  $ \zeta \in I _ \eta $,
for some unique $ \eta <\alpha' _1$; then
put $O' _\zeta = O _{ f_ \eta( \zeta ) } $. 

We shall show that 
$(O'_ \zeta  ) _{ \zeta  \in \alpha'} $
witnesses 
$[ \alpha ' , \alpha '  ]$-\brfrt incompactness of $X$.
Indeed, 
since $(Z_ \eta ) _{ \eta <\alpha' _1} $ is a partition of $\alpha''$, 
then,
by Condition (3) above, and by the definition of 
the
 $W_ \eta$'s,
we get that 
$ \bigcupdot _{ \eta < \alpha '_1} W_ \eta = \alpha  $.
Since each $f_\eta$ is a bijection,
and $\alpha' = \bigcupdot _{ \eta < \alpha '_1} I_ \eta  $,
we get that
$(O'_ \zeta  ) _{ \zeta  \in \alpha'} $
is actually a rearrangement of  $(O_ \delta ) _{ \delta \in \alpha} $,
thus it is still a cover of $X$.

Let $Y \subseteq \alpha'$, and suppose that  
$(O'_ \zeta  ) _{ \zeta  \in Y} $ is a cover of $X$. We have to show that
$Y$ has order type $\alpha'$. It is enough to  show that,
for every $\eta <\alpha' _1$, 
$ |Y \cap I _ \eta | = \kappa $, thus
$ Y \cap I _ \eta $
and
$ I _ \eta $
have the same order type ($= \kappa $).
Hence $Y$ and $\alpha'$  have the same order type,
since $ \alpha '= \bigcup _{ \eta  < \alpha '_1}  I _ \eta $.

So, fix $\eta  < \alpha '_1$. 
For every 
$  \xi \in Z_ \eta $,
 by Condition (5) above,
we have that $x_ \xi \not\in O_ \delta $,
for every $\delta \in [0, \gamma _ \xi) \cup  [\gamma _{ \xi + 1},  \alpha )$.
Since 
$(O'_ \zeta  ) _{ \zeta  \in Y} $ is a cover,
there is 
$\zeta  \in Y$ such that 
$x _ \xi \in O'_ \zeta   $.
Necessarily,
$O' _\zeta = O _ \delta  $,
for some $\delta \in [ \gamma _ \xi, \gamma _{ \xi+ 1} )$, thus
$ \delta \in W_ \eta $, hence 
$\delta= f_ \eta( \zeta )$,
and   $ \zeta \in Y \cap I _ \eta $.
By construction, 
$  |Z_ \eta| = \kappa  $.
Since, for 
$\xi \not= \xi'$,
the intervals 
$[ \gamma _ \xi, \gamma _{\xi+1} )$  
and
$[ \gamma _{\xi'} , \gamma _{\xi'+1} )$ 
are disjoint, then, for each 
$  \xi \in Z_ \eta $,
we get a distinct 
$\delta \in \alpha $, hence a distinct 
 $ \zeta \in Y \cap I _ \eta $,
thus $ | Y \cap I _ \eta| = \kappa $.
 
Since the above argument works for each 
$\eta  < \alpha '_1$, we get that
$(O'_ \zeta  ) _{ \zeta  \in \alpha'} $
is indeed a counterexample to 
$[ \alpha ' , \alpha '  ]$-\brfrt compactness.
\end{proof}

\begin{example} \labbel{exkk}
If $\tau$ is not supposed to be closed under unions, the conclusion in the second
statement in Lemma \ref{nonoltreprod} might fail.

Indeed, let $\kappa $ be an infinite regular cardinals, 
let $X= \kappa \cdot \kappa   $,
and let $\tau$ consist of the sets of the form
$[\kappa \cdot \gamma , \kappa \cdot \gamma  + \delta ]$, 
for $\gamma, \delta  < \lambda $.   
Then $(X, \tau )$ is trivially not 
 $[ \kappa \cdot \kappa  , \kappa \cdot \kappa    ]$-\brfrt compact,
but it is $[ \kappa +1 , \kappa +1 ]$-\brfrt compact,
since any cover of $X$ always remains a cover if we take off
any single member of the cover. 

Actually, if $|\alpha|= \kappa $,
then $X$ is $[\alpha, \alpha ]$-compact if and only if  
$\alpha$ has not the form $\kappa \cdot \alpha_1 $, for some ordinal
$\alpha_1$.   

The example also shows that the assumption that $\tau$ 
is closed under unions is necessary in Condition (5)
in Theorem \ref{cork} below, as well as in Condition (4) 
in Corollary \ref{corw}. 
 \end{example}

As a consequence of Lemma \ref{nonoltreprod},
for spaces of cardinality $\kappa$,
the theory of 
$[ \alpha  , \alpha ]$-\brfrt compactness
becomes trivial on a large class of limit ordinals, 
as explicitly stated in the next Theorem.  
More strikingly,
for countable spaces, the theory of
$[ \alpha  , \alpha ]$-\brfrt compactness
is nontrivial only for ordinals
$ \leq \omega \cdot \omega $ (Corollary \ref{corw} below).

\begin{theorem} \labbel{cork}
If $\kappa$ is an infinite regular cardinal and $|X|= \kappa $, 
then the following conditions are equivalent.
\begin{enumerate} 
\item
 $X$  is  $[\kappa \cdot \kappa, \kappa \cdot \kappa]$-\brfrt compact.
\item 
$X$ is   $[ \alpha  , \alpha ]$-\brfrt compact, for some   limit ordinal $\alpha$ 
 of the form  $ \alpha = \kappa \cdot \alpha _1$,
for some $\alpha_1 > 0$ with $ |\alpha _1| \leq \kappa$. 
\item 
$X$ is   $[ \alpha  , \alpha ]$-\brfrt compact, for  every   ordinal
$\alpha $ of the form $\alpha_1 + \kappa ^ \varepsilon $, 
for some ordinals $\alpha_1 \geq 0$ and  $\varepsilon >1$
such that $\varepsilon$ is either a successor ordinal, or
$\cf \varepsilon = \kappa $.
\item 
$X$ is   $[ \alpha  , \alpha + \kappa \cdot \omega )$-\brfrt compact, for  every   ordinal
$\alpha $ of the form $\alpha_1 + \kappa ^ \varepsilon $, 
for some ordinals $\alpha_1 \geq 0$ and  $\varepsilon >1$
such that $\varepsilon$ is either a successor ordinal, or
$\cf \varepsilon = \kappa $.
\end{enumerate}
If $\tau$ is closed under unions, then the preceding conditions are also equivalent to:  
  \begin{enumerate} 
\item[(5)]
$X$ is   $[ \alpha  , \alpha ]$-\brfrt compact, for some  nonzero  ordinal $\alpha$ 
 such that  $| \alpha | \leq\kappa $. 
  \end{enumerate} 
 \end{theorem}

 \begin{proof} 
 (1) $\Rightarrow $  (2)  and  (1) $\Rightarrow $  (5) are trivial.

(2) $\Rightarrow $ (3) and, for $\tau$ closed under unions, (5) $\Rightarrow $  (3) follow from Lemma \ref{nonoltreprod}. 

(3) $\Rightarrow $  (4) is from Corollary \ref{cortransfer}(3).

(4) $\Rightarrow $  (1) is immediate from Proposition \ref{simple}(1),
with $\alpha_1=0$ and $\varepsilon=2$.  
\end{proof}  

\begin{corollary} \labbel{corkk}
Suppose that $\kappa$ is an infinite regular cardinal, $|X|= \kappa $,
and let $A$ be the set of all ordinals $\alpha < \kappa ^+$ of  the form  
$ \kappa \cdot ( \alpha ^\bullet + n)$,
with 
$\cf \alpha ^\bullet = \kappa $
and $ 0 \leq n< \omega$.

Then
$X$ is   $[ \alpha  , \alpha ]$-\brfrt compact,
for some $\alpha \in A$, if and only if
$X$ is   $[ \alpha  , \alpha ]$-\brfrt compact,
for all $\alpha \in A$.
 \end{corollary}

 \begin{proof} 
Suppose that $X$ is   $[ \alpha'  , \alpha' ]$-\brfrt compact,
for some $\alpha' \in A$. Since  $\alpha'$  
is of the form given in Clause \ref{cork}(2),
then all the equivalent conditions in Theorem \ref{cork} hold.

Now let $\alpha = \kappa \cdot ( \alpha ^\bullet + n) \in A$ be arbitrary. 
Since $\cf \alpha ^\bullet = \kappa $,
then $\alpha ^\bullet = \kappa \cdot \alpha ^*$,
where $\alpha^*$ is either successor or has 
cofinality $\kappa$ itself. In both cases,
$\alpha= \kappa \cdot ( \alpha ^\bullet + n)$ is of the form
$ \alpha _1 + \kappa ^ \varepsilon + \kappa \cdot n$,
with $\varepsilon > 1$  either 
successor, or of cofinality $\kappa$.
Thus, $X$ is   $[ \alpha  , \alpha]$-\brfrt compact,
 in force of Clause
\ref{cork}(4) and of 
Proposition \ref{simple}(1). 
\end{proof}

\begin{corollary} \labbel{corw}
If $|X|= \omega  $, 
then the following conditions are equivalent.
\begin{enumerate} 
\item
 $X$  is  $[ \omega  \cdot \omega , \omega  \cdot \omega ]$-\brfrt compact.
\item 
$X$ is   $[ \alpha  , \alpha ]$-\brfrt compact, for some countable  limit ordinal $\alpha$. 
\item
 $X$  is  $[ \omega  \cdot \omega , \infty) $-\brfrt compact.
\end{enumerate}
If $\tau$ is closed under unions, then the preceding conditions are also equivalent to:
  \begin{enumerate} 
\item[(4)]
$X$ is   $[ \alpha  , \alpha ]$-\brfrt compact, for some  nonzero ordinal $\alpha< \omega _1$.
  \end{enumerate} 
 \end{corollary}

 \begin{proof} 
The equivalence of  (1), (2),  and  (4) is a particular case of 
Theorem  \ref{cork} (Conditions (1), (2) and (5) there). 

(3) $\Rightarrow $  (1) is immediate from Proposition \ref{simple}.

In order to finish the proof, suppose that (2) holds. Then, by 
Theorem  \ref{cork} (2) $\Rightarrow $  (4), $X$ is
 $[ \delta  , \delta ]$-\brfrt compact, for every ordinal $\delta$ 
of the form $ \alpha _1 + \omega ^ \varepsilon + m$, for $\varepsilon>1$,
that is, for every countable 
 ordinal 
$\delta \geq \omega \cdot \omega $.
Since $X$, being countable,  is trivially  $[ \delta  , \delta ]$-\brfrt compact
for every uncountable $\delta$, we get  
 $[ \omega \cdot \omega  , \infty )$-\brfrt compactness
from Proposition \ref{simple}(2). Hence (3) holds.
\end{proof}  

A result similar to Corollary \ref{corw} 
holds for $T_1$ spaces (of arbitrary cardinality): see
Corollary \ref{t1cor}.

\begin{corollary} \labbel{lindw}
If $| X|= \omega $, then the Lindel\"of ordinal of $X$ 
is either $ \omega_1$, or is $ \leq \omega \cdot \omega $.  

More generally, if $\kappa$ is regular, and 
$|X|= \kappa $, then the Lindel\"of ordinal 
of $X$ cannot have  the form  $\alpha_1 + \kappa ^ \varepsilon + \gamma $, 
with $0 <\gamma < \kappa \cdot \omega $, and  $\varepsilon >1$
such that $\varepsilon$ is either a successor ordinal, or
$\cf \varepsilon = \kappa $.
 \end{corollary} 

\begin{proof}
The first statement is immediate from Corollary \ref{corw} (2) $\Rightarrow $  (3).
 
As for the second statement, if the Lindel\"of ordinal of $X$ is $<\kappa^+$, 
then $X$ is 
 $[ \alpha  , \alpha  ]$-\brfrt compact, for some $\alpha$ as in 
Item (2) in Theorem \ref{cork}. The conclusion now follows from Proposition 
\ref{simple} and  Item (4) in Theorem \ref{cork}.
\end{proof}

\section{An exact characterization of transfer properties} \labbel{exact}

In this section we introduce some further examples,  more involved than those presented in Examples \ref{exex}. This is necessary in order to avoid the limitations
given by Theorem \ref{cork} and Corollaries \ref{corkk} and  \ref{corw}.
The examples introduced in this section are optimal, in the sense that they
provide an exact characterization of those ordinals $\alpha$ and $\beta$ such that
 $[ \alpha  , \alpha ]$-\brfrt compactness
implies $[ \beta   , \beta  ]$-\brfrt compactness.

\begin{definitions} \labbel{saa}
As usual, we denote by ${^ \alpha 2}$ the set of 
all the functions from $\alpha$ to $2= \{ 0, 1\} $.

If $ f \in {^ \alpha 2}$, the \emph{support} of $f$ 
is $\{ \delta \in \alpha \mid f( \delta )=1\}$.   

For nonzero ordinals $\beta \leq \alpha $, we now define 
$S_ \beta  ( \alpha ) = 
\{ f \in {^ \alpha 2} \mid \text{the support of }\brfr f \text{ has order type } < \beta \}$. 

$S_ \beta  ( \alpha ) $ is in a one-to one correspondence,
via characteristic functions, with the 
set of all subsets of $\alpha$ which have order type $<\beta$.
The $S$ in our notation is a reminder for \emph{Subset}.  
However, 
in the present note, we shall mainly
deal with elements of $ {^ \alpha 2}$,
rather than with subsets of $\alpha$,
since it will be more convenient
for our purposes.

 We shall mainly deal with the case $\beta= \alpha $,
and we shall consider various families of subsets of
$S_ \beta  ( \alpha ) $. 

We put
$X( \beta, \alpha )= ( S_ \beta  ( \alpha ) , \tau_0  )  $,
 where the elements of $\tau_0$ are all the subsets
of $ S_ \beta  ( \alpha )$ having the form
$Z( \varepsilon) =\{ f \in  S_ \beta  ( \alpha ) \mid f( \varepsilon ) =0\}$,
$\varepsilon$ varying in $\alpha$.  

We also let $X_U( \beta, \alpha )= ( S_ \beta  ( \alpha ) , \tau_{_U}  )  $,
 where $\tau_{_U}$ is the smallest family of subsets of 
$S_ \beta  ( \alpha ) $ which contains 
 $\tau_0$ above, and is closed under unions. In other words,
a generic element of $\tau_{_U}$ has the form
$ \bigcup _{ \varepsilon \in H} Z( \varepsilon ) =
\{ f \in S_ \beta  ( \alpha ) \mid  f( \varepsilon ) =0 \text{, for some } \varepsilon \in H  \} $,
 for some $H \subseteq \alpha $. 
 
For $\alpha \leq 2$
and $\beta >1$, 
 neither  $\tau_0$ nor  $\tau_{_U}$ are topologies, since they are not closed
under finite intersections. However, if we take the closure of 
$\tau_{_U}$ under finite intersections, we do get a topology $\tau$
on $S_ \beta  ( \alpha )$. 
For 
${\barb{\varepsilon} }= \{ \varepsilon_0, \varepsilon _1, \dots, \varepsilon _{n-1} \} 
 \in S _{n+1}( \alpha )  $, 
let $Z({\barb{\varepsilon} })=Z( \varepsilon_0, \varepsilon _1, \dots, \varepsilon _{n-1} ) =
Z( \varepsilon_0) \cap  Z( \varepsilon _1) \cap \dots \cap Z(\varepsilon _{n-1} )=
\{ f \in  S_ \beta  ( \alpha ) \mid 
f( \varepsilon_0 ) = f(\varepsilon _1)= \dots = f(\varepsilon _{n-1} )= 0\}$.
Members of $\tau$ have then the  form 
$ \bigcup _{ \barb{\varepsilon} \in H } Z( \barb{\varepsilon} ) $,
$H$ varying among the subsets of $S _ \omega ( \alpha )$.
We let $X_ \tau ( \beta, \alpha )= ( S_ \beta  ( \alpha ) , \tau)  $.

The above topology $\tau$ is $T_0$, but not even $T_1$. 
A topology satisfying stronger separation axioms can be introduced
as follows. 

$X_T( \beta, \alpha )= ( S_ \beta  ( \alpha ) , \tau_{_T}  )  $,
where $ \tau_{_T}$ is the (Tychonoff) topology inherited by the product 
topology on $ {{^ \alpha 2}}$, where $2$ is given the discrete topology.   
Notice that $X_T( \beta, \alpha )$ inherits from $ {{^ \alpha 2}}$ also the 
structure of a  topological group,

We shall write $X( \beta ) $ in place of $X( \beta , \beta )$, and similarly for 
$X_U( \beta )$, $X_ \tau ( \beta  )$, and $X_T( \beta )$.
The subscript $\tau$ is a reminder for \emph{topology}, the subscript $U$ is a reminder for (closed under) \emph{Unions}, and the  
subscript $T$ is a reminder for  \emph{Tychonoff}.
 \end{definitions}   

\begin{remark} \labbel{precedenti}
Similar constructions, when restricted to cardinal numbers, have sometimes been considered in the literature. See, e. g.,  \cite[Example 4.1]{AB}, 
\cite{topproc}
and \cite[Example 4.2]{St}.
 \end{remark}   

\begin{lemma} \labbel{covers}
Suppose $0<\beta \leq \alpha  $, and assume the notations in Definition  \ref{saa}.

If $H \subseteq \alpha $, then the sequence
$(Z( \varepsilon)) _{ \varepsilon \in H }$
is a cover of  $X( \beta, \alpha )$ if and only if 
$H$ has order type $ \geq\beta$.
In particular, $X( \beta, \alpha )$ is not 
$[ \beta, \beta ]$-\brfrt compact, hence neither
$X_U( \beta, \alpha )$, nor $X_ \tau ( \beta, \alpha )$, nor 
$X_T( \beta, \alpha )$ are $[ \beta, \beta ]$-\brfrt compact. 
\end{lemma} 

\begin{proof} 
If $H$ has order type $ <\beta$,
define $f: \alpha \to 2$
by $f ( \delta )= 1$ if and only if
$ \delta \in H$.
Then $f \in X( \beta, \alpha )$, but 
$f$ belongs to no $Z ( \varepsilon ) $ ($\varepsilon \in H$).    

On the contrary, 
suppose by contradiction that $H$ has order type $ \geq\beta$,
but there is $f \in X( \beta, \alpha )$ such that 
$f$ belongs to no $Z ( \varepsilon ) $ ($\varepsilon \in H$). 
If $f \not\in Z ( \varepsilon )$, then $f( \varepsilon )= 1$,
thus the support of $f$ contains $H$, which has order type    
$ \geq\beta$, and this contradicts $f \in X( \beta, \alpha )$.

In order to show that  $X( \beta, \alpha )$ is not 
$[ \beta, \beta ]$-\brfrt compact,
it is enough to choose some $H \subseteq \alpha$
of order type $\beta$. 
Then, by above,  $(Z( \varepsilon)) _{ \varepsilon \in H }$
is a cover of  $X( \beta, \alpha )$, but 
if $K \subseteq H$ has order type $<\beta$,
then $(Z( \varepsilon)) _{ \varepsilon \in K }$ 
is not a cover of  $X( \beta, \alpha )$.
The same argument works for   
 $X_U( \beta, \alpha )$, $X_ \tau ( \beta, \alpha )$,  and 
$X_T( \beta, \alpha )$.
\end{proof}

\begin{theorem} \labbel{iff1}
Let $\alpha$ and  $\beta$ be nonzero ordinals, and
assume the notations in Definition  \ref{saa}.
Then the following conditions are equivalent.
      \begin{enumerate}  
      \item[(a)]
$X( \beta ) $ is not $[ \alpha, \alpha ]$-\brfrt compact.
       \item[(b)]
There exists an injective function
$f: \beta  \to \alpha  $ 
such that, for every  $K \subseteq \alpha  $ with order type 
$<\alpha$, it happens that $f ^{-1} (K)$ has order type $< \beta  $.
       \item[(c)]
For arbitrary $(X, \tau )$, $[ \alpha, \alpha ]$-\brfrt compactness implies
$[ \beta, \beta ]$-\brfrt compactness.
       \end{enumerate}   
 \end{theorem} 

\begin{proof} 
 (a) $\Rightarrow $  (b). Suppose that (a) holds. 
Then $X( \beta )$ has a cover 
$(O_ \delta ) _{ \delta \in \alpha } $
such that, whenever $H \subseteq \alpha$ has order type
$<\alpha$, then
$(O_ \delta ) _{ \delta \in H } $
is not a cover of $X( \beta )$.
By Lemma \ref{irredundant}, we can suppose that
$O_ \delta \not= O _{ \delta '} $, for 
$\delta \not= \delta ' \in \alpha $.
 Because of the definition of $\tau_0$, 
for each $\delta \in \alpha $, there is 
$\varepsilon \in \beta $ such that  
$O_ \delta = Z( \varepsilon )$. 
Let $W= \{ \varepsilon \in \beta \mid 
Z( \varepsilon )=O_ \delta \text{, for some } \delta \in \alpha \} \subseteq \beta $.
Since 
$(O_ \delta ) _{ \delta \in \alpha } $ 
is a cover of $X( \beta )$,
then also 
$ (Z( \varepsilon )) _{ \varepsilon \in W} $ 
is a cover of $X( \beta )$.
By Lemma \ref{covers}, $W$ has order type $\beta$.

Let $g: W \to \alpha $ be defined by 
$g( \varepsilon )= \delta $ if and only if 
$Z( \varepsilon )=O_ \delta$. 
Such a $\delta$ exists because of the definition of $W$,
and is unique because of the property   
 $(O_ \delta ) _{ \delta \in \alpha } $ is assumed to satisfy.

If  $K \subseteq \alpha  $ has order type 
$<\alpha$, then, by 
$[ \alpha, \alpha ]$-\brfrt incompactness,
 $(O_ \delta ) _{ \delta \in K} $ is not a cover of
$X( \beta )$.
Hence, $ (Z( \varepsilon )) _{ \varepsilon \in g^{-1}(K) } $
is not a cover of $X( \beta )$.
 By Lemma \ref{covers}, $g^{-1}(K)$ has order type $<\beta$.

Thus, the counterimage by $g$ of a subset of $ \alpha $
of order type $<\alpha$  has order type $<\beta$.
Since $W$ has order type $\beta$, then, by composing $g$
with an isomorphism between $W$ and $ \beta $, we get a function $f$
satisfying the required property.  
Notice that $g$ (hence also $f$) is injective, since
$Z( \varepsilon ) \not = Z( \varepsilon ')$, for 
$ \varepsilon  \not =  \varepsilon '$.

(b) $\Rightarrow $  (c) is a particular case of Proposition \ref{transfer}.

(c) $\Rightarrow $  (a). If (c) holds, then
$X( \beta )$
is not
$[ \alpha  , \alpha   ]$-\brfrt compact,
since, by Lemma \ref{covers},
 it is not  
$[ \beta , \beta  ]$-\brfrt compact.
\end{proof} 
 
\begin{remark} \labbel{w+ww2}  
Thus, for example, for every pair 
$\nu \leq  \kappa $ of infinite regular
cardinals, $ [ \kappa + \kappa , \kappa  + \kappa  ] $-\brfrt compactness
does not imply 
$ [ \kappa \cdot  \nu, \kappa \cdot  \nu ] $-\brfrt compactness,
since there is no function 
$f: \kappa \cdot \nu  \to \kappa  + \kappa  $
satisfying Condition (b) in Theorem \ref{iff1}.  

Similarly,
$ [ \kappa^2 + \kappa , \kappa^2  + \kappa  ] $-\brfrt compactness
does not imply 
$ [ \kappa \cdot  \nu , \kappa \cdot  \nu] $-\brfrt compactness.

Thus, Corollary \ref{cortransfer}(2)(3) cannot be improved. Notice that, because of  
Theorem \ref{cork}(2) $\Rightarrow $  (1), 
 if $X$ is 
$ [ \kappa  + \kappa , \kappa  + \kappa  ] $-\brfrt compact 
and not
$ [ \kappa ^2, \kappa ^2] $-\brfrt compact,
 then $|X| > \kappa  $. 
 \end{remark} 

\begin{corollary} \labbel{coriff1} 
Suppose that $\alpha$ and $\beta$ are nonzero ordinals,
and $| \alpha | \not =| \beta |$. Then the following statements hold. 
\begin{enumerate} 
\item 
$X( \beta )$ is  
$ [ \alpha  , \alpha   ] $-compact.
\item
There is some  $(X, \tau )$
which is $ [ \beta , \beta  ] $-compact
and not  
 $ [ \alpha  , \alpha   ] $-compact.
\end{enumerate}  
\end{corollary}

 \begin{proof}
If  
$f: \beta  \to \alpha  $  is an injective function, 
then $| \alpha |  >| \beta |$, 
since $| \alpha | \not =| \beta |$.
Hence 
  $K = f( \beta )\subseteq \alpha  $ has order type 
$<\alpha$, but $f ^{-1} (K) = \beta $ has order type $ \beta  $.
Hence Condition (b) in Theorem \ref{iff1} fails, hence also the equivalent
Conditions (a) and (c) fail.
 \end{proof}  

Of course, Corollary \ref{coriff1} does not hold in the case when $\tau$ is 
requested to be closed under unions. See, e. g., Corollary  \ref{cortransfer}(6)-(8). 
The next Theorem is the analogue of Theorem \ref{iff1} 
in the case when $\tau$ is asked to be closed under unions.

\begin{theorem} \labbel{iff2}
Let $\alpha$, $\beta$ be nonzero ordinals, and
assume the notations in Definition  \ref{saa}.
Then the  following conditions are equivalent.
      \begin{enumerate}  
      \item[(a)]
$X_U( \beta ) $ is not $[ \alpha, \alpha ]$-\brfrt compact.
       \item[(b)]
There exists a  function
$f: \beta  \to \alpha  $ 
such that, for every  $K \subseteq \alpha  $ with order type 
$<\alpha$, it happens that $f ^{-1} (K)$ has order type $< \beta  $.
       \item[(c)]
For every $X$ and $\tau$, if  $ \tau $ is closed under unions, then $[ \alpha, \alpha ]$-\brfrt compactness 
of $(X, \tau )$
implies
$[ \beta, \beta ]$-\brfrt compactness
of $(X, \tau )$.
       \end{enumerate}   
 \end{theorem} 

\begin{proof}
 (a) $\Rightarrow $  (b).
Suppose that (a) holds,
and that
$(O_ \delta ) _{ \delta \in \alpha } $
is a counterexample to the
$[ \alpha , \alpha ]$-\brfrt compactness
of
$X( \beta )$.
By the definition of 
$ \tau_{_U}$, each $O_ \delta $ 
has the form  
$ \bigcup _{ \varepsilon \in W _ \delta } Z( \varepsilon )$,
 for some $ W_ \delta  \subseteq \beta $. 

For $\delta \in \alpha $, let
$W _ \delta ^* = W_ \delta  \setminus \bigcup _{ \gamma < \delta }  W _ \gamma   $,
and let 
$O_ \delta ^*= \bigcup _{ \varepsilon \in W ^* _ \delta } Z( \varepsilon )$.
Notice that $(O^*_ \delta ) _{ \delta \in \alpha } $
is still a cover of $X( \beta )$, hence it is still a counterexample to the
$[ \alpha , \alpha ]$-\brfrt compactness
of
$X( \beta )$,
since $O^*_ \delta \subseteq O _ \delta $, for every  $  \delta \in \alpha  $.
 
Since $(O_ \delta^* ) _{ \delta \in \alpha } $  covers $X( \beta )$,
we have 
$   \bigcup \{ Z ( \varepsilon ) \mid 
 \varepsilon \in W ^*_ \delta \text{, for some  } \delta \in \alpha  \} 
=\bigcup \{ Z ( \varepsilon ) \mid 
 \varepsilon \in \bigcup _{ \delta \in \alpha } W^* _ \delta  \}=
 X( \beta ) $, hence, by Lemma \ref{covers},  
the order type of $W=  \bigcup _{ \delta \in \alpha } W ^*_ \delta =
\bigcup _{ \delta \in \alpha } W _ \delta$ 
 equals $\beta$. 
  
Let $g: W \to \alpha $ be defined by 
$g( \varepsilon )= $ the unique $  \delta \in \alpha $
such that $  \varepsilon \in W^* _\delta   $.
If  $K \subseteq \alpha  $ has order type 
$<\alpha$, then, by 
$[ \alpha, \alpha ]$-\brfrt incompactness,
 $(O^*_ \delta ) _{ \delta \in K} $ is not a cover of
$X( \beta )$.
Hence $ (Z( \varepsilon )) _{ \varepsilon \in g^{-1}(K) } $
is not a cover of $X( \beta )$.
 By Lemma \ref{covers}, $g^{-1}(K)$ has order type $<\beta$.

We have proved that the counterimage by $g$ of a subset of $\alpha$ 
of order type $<\alpha$  has order type $<\beta$, thus, arguing as in 
corresponding part of the proof of 
Theorem \ref{iff1}, and since $W$ has order type $\beta$,   we get a function $f$
as desired.  
 
(b) $\Rightarrow $  (c) follows from the last statement in Proposition \ref{transfer}.

(c) $\Rightarrow $  (a). If (c) holds, then
$X_U( \beta )$
is not
$[ \alpha  , \alpha   ]$-\brfrt compact,
since, by Lemma \ref{covers},
 it is not  
$[ \beta , \beta  ]$-\brfrt compact, 
and since $\tau _{_U} $ is closed under unions.
\end{proof}

\section{$[\alpha, \beta ]$-\brfrt compactness of $T_1$ spaces} \labbel{t1sec}

The counterexamples presented in Examples \ref{exex}(2)  and  \ref{exgen} satisfy very few separation axioms.
In fact, we can show that more results about
$[ \beta  , \alpha ]$-\brfrt compactness can be proved just on the assumption
that we are dealing with $T_1$ topological spaces.
Indeed, since  in this note we have kept the greatest possible generality, 
we mention that we do not actually need a $T_1$ topological space, in order to prove
 the results in the present section. The following weaker notion is enough.

\begin{definition} \labbel{t1surrogate}
If $X$ is a nonempty set,
and $\tau$ is a nonempty family of subsets of $X$, we say that 
$(X, \tau)$ is $T_1$ if and only if, 
for every $O \in \tau$, and every $x \in O$,
$O \setminus  \{ x\} \in \tau  $.

Clearly, the above condition is equivalent to asking that,
for every $O \in \tau$, and every finite $ F \subseteq X$,
$O \setminus  F \in \tau  $. 
Trivially, if $\tau$ is a topology on $X$, then 
$(X, \tau)$ is $T_1$ in the above sense if and only if 
it is $T_1$ in the ordinary topological theoretical sense.
\end{definition}   

It is convenient to introduce some notation, in order to state the next Proposition 
more concisely.

\begin{definition} \labbel{alphaell}    
If $ \beta $ is an infinite ordinal, we let $ \beta  ^{\ell}$ be the largest limit
ordinal $ \leq \beta $. Thus, $ \beta  ^{\ell}= \beta  - n$, for an appropriate 
$n \in \omega$.
\end{definition}

\begin{proposition} \labbel{basict1} 
Suppose that $X$ is  $T_1$, and let $\alpha$ be an infinite ordinal.
\begin{enumerate}
\item
$X$ is $[ \alpha , \alpha ]$-\brfrt compact if and only if  
$X$ is $[ \alpha +1, \alpha+1 ]$-\brfrt compact. 
\item
  For every $ n \in \omega$ and infinite $\beta \leq \alpha $, 
 $X$ is $[ \beta  , \alpha ]$-\brfrt compact
if and only if it is $[ \beta  ^{\ell} , \alpha+n ]$-\brfrt compact.
\item
  For every  infinite $\beta \leq \alpha $, 
 $X$ is $[ \beta  , \alpha ]$-\brfrt compact
if and only if it is $[ \beta  ^{\ell} , \alpha+ \omega ) $-\brfrt compact.
\item
If $\beta \leq \alpha $ and $\beta$ is infinite,
 then
 $X$ is $[ \beta  , \alpha ]$-\brfrt compact
 if and only if it is $[ \gamma  , \gamma  ]$-\brfrt compact,
for every limit ordinal $\gamma$ with $ \beta  ^{\ell} \leq \gamma \leq \alpha  $.
\end{enumerate}   
\end{proposition}

 \begin{proof} 
(1)  One implication follows from Corollary \ref{cortransfer}(1)
and Proposition \ref{simple}(1). 

On the other hand, suppose that $X$ is
$[ \alpha +1, \alpha+1 ]$-\brfrt compact and let
$(O_ \delta ) _{ \delta \in \alpha } $ be a cover of $X$.
Without loss of generality, e. g., by Lemma \ref{irredundant},
we can suppose that 
$O_0 \not= \emptyset $. 
Let $x \in O_0$, 
and, for $ \delta \in \alpha$ with
$ \delta > 0$,
let 
$O'_ \delta = O _ \delta \setminus \{ x \} $.
Since $(X, \tau)$ is assumed to be $T_1$,
each $O'_ \delta $ still belongs to $\tau$.
Moreover,   
$(O'_ \delta ) _{ \delta \in \alpha } $ is still a cover of $X$.
Notice that every subcover of 
$(O'_ \delta ) _{ \delta \in \alpha } $ 
must contain 
$O_0$, which is the only element of the cover containing 
$x$. 

Rearrange $(O'_ \delta ) _{ \delta \in \alpha } $
as $(U_ \delta ) _{ \delta \in \alpha +1} $ by letting
$U_ \delta =O' _{ f(\delta)}$, where
$f: \alpha +1\to \alpha $ is the bijection defined by
\[ 
f( \delta) =\begin{cases}
\delta +1 &    \text{if  $\delta < \omega   $},\\
\delta &    \text{if  $ \omega \leq \delta < \alpha    $},\\
 0 &    \text{if  $\delta = \alpha  $}.\\
\end{cases}
\] 

By applying $[ \alpha +1, \alpha+1 ]$-\brfrt compactness
to 
$(U_ \delta ) _{ \delta \in \alpha +1} $,
we get 
$H \subseteq  \alpha +1$
such that $H$ has order type $<\alpha+1$,
and    $(U_ \delta ) _{ \delta \in H} $ is a cover.
Since $U_ \alpha =O'_0$,
and $O'_0$ is the only element of the cover
containing $x$, we have that   
 $U_ \alpha$ belongs to the subcover, that is,
$\alpha \in H$. 
Since $H$ has order type $<\alpha+1$,
then necessarily 
$H \cap \alpha $ has order type $<\alpha$.
Since $f _{| \alpha } $ is order-preserving,
then also  
$ f ^{-1} (H \cap \alpha) $
has order type $<\alpha$.
Hence 
$ K= f ^{-1} (H ) $, too, 
has order type $<\alpha$,
since $\alpha$ is infinite, and we are adding
to $ f ^{-1} (H \cap \alpha) $
just one element 
``at the beginning''. 

Then 
$(O'_ \delta ) _{ \delta \in K } $ is a cover of $X$
indexed by a set of order type $<\alpha$,
and also $(O_ \delta ) _{ \delta \in K } $ is a cover,
since 
$O' _ \delta \subseteq O _ \delta $,
for every  $ \delta \in \alpha$.
Hence,
$(O_ \delta ) _{ \delta \in K } $ is a subcover
of order type $<\alpha$ of our original cover
$(O_ \delta ) _{ \delta \in \alpha } $,
and we have proved  
$[ \alpha , \alpha]$-\brfrt compactness.

(2) - (4) are immediate from (1) and Proposition \ref{simple}.
\end{proof} 

Of course, Item 1 in Proposition \ref{basict1}
is false without the assumption that $\alpha$ is infinite. Indeed, the discrete space with
exactly $n$ elements is $[n+1,n+1]$-\brfrt compact, but not $[n,n]$-\brfrt compact.

The next Lemma captures a very useful consequence of 
being $T_1$. 

\begin{lemma} \labbel{t1lem1} 
Suppose that $\alpha$ is an ordinal, $ \cf \alpha= \omega $,
and $(\alpha_n) _{n \in \omega } $
is a strictly increasing sequence 
such that
$\sup_{n \in \omega } \alpha_n = \alpha  $.

If $X$ is $T_1$ and not 
$[ \alpha , \alpha ]$-\brfrt compact,
then there is a counterexample 
$(O_ \delta ) _{ \delta \in \alpha} $  
to the
$[ \alpha , \alpha ]$-\brfrt compactness of $X$ 
with the property that, for every $ n \in \omega$,
$O _{ \alpha _n} $ is indispensable (Definition \ref{indisp}).
\end{lemma} 

\begin{proof}
Let $\alpha$ and the $\alpha_n$'s be given. 
Suppose that
$(O_ \delta ) _{ \delta \in \alpha } $  
is a counterexample to
$[ \alpha, \alpha ]$-\brfrt compactness.
By Lemma \ref{irredundant},
we can also suppose that,
for every $\delta <  \alpha $, 
$O_ \delta $ is not contained in 
 $ \bigcup  _{ \varepsilon < \delta  } O_ \varepsilon  $.
In  particular, for every 
$ n \in \omega $,
we can choose 
$x_n \in O _{ \alpha _n} $ such that
 $ x_n \not\in \bigcup  _{ \varepsilon < \alpha _n  } O_ \varepsilon  $.
Define $(O'_ \delta ) _{ \delta \in \alpha } $  as follows.
\[ 
O'_ \delta =\begin{cases}
O_ \delta &    \text{if  $ \delta \leq \alpha _0$},\\
O_ \delta \setminus \{ x_0, \dots, x_n \} &    \text{if  $  \alpha _n < \delta \leq \alpha _{n+1}  $}.\\
\end{cases}
\] 
Since $X$ is $T_1$,
each $O'_ \delta $ still belongs to $\tau$.
Moreover, 
$(O'_ \delta ) _{ \delta \in \alpha } $
is still a cover of $X$. Indeed,
for every $ n \in \omega$,  $x_n \in O' _{ \alpha _n} $. If $x$ is not
one of the  $x_n$'s, then $x \in O_ \delta $, for some $ \delta \in  \alpha $,
 and also  
$x \in O' _ \delta $. 
Since 
$O'_ \delta \subseteq O_ \delta $,
for every $\delta \in \alpha $,  
we have that
$(O'_ \delta ) _{ \delta \in  \alpha } $, too,
is a counterexample to
$[ \alpha, \alpha ]$-\brfrt compactness, and
it is easy to see that  
$(O' _{ \alpha _n} ) _{ n \in \omega  } $
 is a set of indispensable elements. Thus, 
$(O'_ \delta ) _{ \delta \in  \alpha } $ is a cover as wanted.
 \end{proof} 

Many results on $T_1$ spaces will be obtained be rearranging
the indispensable elements given by Lemma \ref{t1lem1}.

The following notation shall be useful in the proof of 
the forthcoming Theorem  \ref{t1}.

\begin{definition} \labbel{bestaast}
If $\beta$ is any ordinal, let  $\beta ^{*} $
  be the smallest ordinal $ \leq \beta$  
such that $|[ \beta ^{*} , \beta ]| \leq \omega $.
Thus, $\beta ^{*} $ is the largest ordinal $ \leq \beta$
which is either $ 0$, or  has uncountable cofinality,
or has cofinality $ \omega$ but can be written as a limit of ordinals of uncountable cofinality.
 \end{definition}   

\begin{theorem} \labbel{t1}
Suppose that  $X$ is  $T_1$, and $ \beta $ is an ordinal of
cofinality $ \omega$. Then
the following conditions are equivalent.
\begin{enumerate}
\item
$X$  is  $[ \beta , \beta   ]$-\brfrt compact.
\item
$X$  is  $[ \beta  +\alpha, \beta   +\alpha ]$-\brfrt compact, for
every ordinal  $\alpha $ with $ | \alpha | \leq  \omega $.  
\item
$X$  is  $[ \beta  +\alpha, \beta   +\alpha ]$-\brfrt compact, for
some ordinal  $\alpha $ with $ | \alpha | \leq  \omega $.
\item
$X$  is  $[ \beta ,  \beta   + \omega _1 )$-\brfrt compact.
 \end{enumerate}  
\end{theorem}  

\begin{proof}
(2) $\Leftrightarrow $  (4) follows from Proposition \ref{simple}(4),
hence it is enough to prove the equivalence of (1) - (3).

We shall first prove the theorem  in some particular cases.

\begin{claim1} \labbel{cl1}
Conditions (1) - (3) are equivalent in case 
 $ \beta = \beta ^{*} + \omega $. 
 \end{claim1}   

 \begin{proof}[Proof of Claim 1]
In case $\beta^*=0$, 
(1) $\Rightarrow $  (2) follows from Proposition \ref{simple}(1)
and Corollary \ref{cortransfer}(4) with $\beta= \gamma = \omega $ .

In case $\beta^*>0$, 
(1) $\Rightarrow $  (2) follows from Proposition \ref{simple}(4)
and Corollary \ref{cortransfer}(5), by taking there
$\alpha= \beta ^*$, $\lambda= \omega $ and $\beta= \beta ^* + \omega $.

(2) $\Rightarrow $  (3)  is trivial.

We shall prove (3) $\Rightarrow $  (1) by proving the contrapositive form.

So suppose that 
$X$  is not  $[ \beta  , \beta  ]$-\brfrt compact,
and $\alpha< \omega_1$. We want to show that
$X$  
is not  $[ \beta  + \alpha , \beta  + \alpha  ]$-\brfrt compact.
For $n < \omega$, let $\alpha_n=  \beta ^* + n$.
Since $\beta= \beta^*   + \omega $, then, by  Lemma \ref{t1lem1},    
there is  some cover
$(O_ \delta ) _{ \delta \in  \beta  } $
witnessing 
$[ \beta   , \beta  ]$-\brfrt incompactness,
and such that each 
$O _{ \alpha _n}$ is indispensable.
If $\beta^*=0$, then 
$[ \beta  + \alpha  , \beta + \alpha  ]$-\brfrt incompactness
follows from Proposition \ref{ord=card}(1), hence in what follows
let us suppose $\beta^*>0$.

For  every $H \subseteq \beta = \beta^*   + \omega$,
if  $(O_ \delta ) _{ \delta \in  H} $
is a cover of $X$, then 
the order type of $H$ is $ \beta =\beta ^* + \omega$,
hence the order type of 
  $H \cap \beta ^* $
 is $\beta ^* $,
since $\beta ^* $ is a limit ordinal.
Moreover, 
  $H \cap [\beta ^*, \beta   ) = [\beta ^*, \beta )$,
since $O_ \delta $ is indispensable,
for every $\delta \in [\beta ^*, \beta   )$.

Let $f:  \beta ^* + \omega + \alpha \to \beta ^* + \omega $   
be a bijection which is the identity on $ \beta ^* $,
and let 
$(U_ \varepsilon  ) _{ \varepsilon  \in  \beta ^* + \omega + \alpha } $
be defined by 
$U _ \varepsilon = O _{ f  ( \varepsilon ) } $. 
We claim that 
$(U_ \varepsilon  ) _{ \varepsilon  \in  \beta ^* + \omega + \alpha } $
witnesses  that
$X$  
is not  $[ \beta ^* + \omega + \alpha , \beta^*   + \omega + \alpha  ]$-\brfrt compact,
and this is what we want, since 
$ \beta ^* + \omega + \alpha  = \beta + \alpha $.
Indeed, if 
$ K \subseteq  \beta^*   + \omega + \alpha  $, and 
$(U_ \varepsilon  ) _{ \varepsilon  \in  K} $
is a cover of $X$, then 
$(O_ \delta ) _{ \delta \in  H} $, with $H = f (K) $, 
is a cover of $X$. 
Since $f$ 
is the identity on $ \beta ^* $, then, by the above mentioned properties
of $H$, we get that 
the order type of 
  $K \cap \beta ^* $
equals the order type of 
$H \cap \beta ^* $,
that
 is, $\beta ^* $;
moreover, 
$K \cap [\beta ^*, \beta ^*+ \omega + \alpha ) = [\beta ^*, \beta ^*+ \omega + \alpha )$,
 thus $K$ has order type
$\beta ^*+ \omega + \alpha $,
hence  $[ \beta ^* + \omega + \alpha , \beta^*   + \omega + \alpha  ]$-\brfrt incompactness
is proved.
 \qedhere$_ {Claim\ 1}$ 
 \end{proof}

\begin{claim2} \labbel{cl2}
Conditions (1) - (3) are equivalent in the case 
 when $\beta ^{*} $ has cofinality $ \omega$, 
and  $ \beta = \beta ^{*} $.
 \end{claim2}   

\begin{proof}[Proof of Claim 2]
In view of  Claim 1, and of 
Proposition \ref{basict1}(1), it is
enough to show that if 
 $ \cf \beta ^{*} = \omega $, then 
 $[ \beta ^*, \beta ^*  ]$-\brfrt compactness
is equivalent to 
 $[ \beta ^* + \omega , \beta ^* + \omega  ]$-\brfrt compactness.
The former implies the latter because of Corollary 
\ref{cortransfer}(3) (taking $\beta= \alpha = \beta ^*$ there),
by Proposition \ref{simple}(4), and  
 since we have assumed that
$\cf \beta ^{*} = \omega$. 
We shall prove the reverse implication by contraposition.
Suppose that $X$
is not
 $[ \beta ^* ,  \beta ^*  ]$-\brfrt compact.
We want  to show that $X$ is not
  $[ \beta ^* + \omega ,  \beta ^* + \omega  ]$-\brfrt compact,
Choose some  strictly increasing sequence 
$(\alpha_n) _{n \in \omega } $
cofinal in $\beta ^* $. This is possible, since
$\cf \beta ^{*} = \omega$.
By Lemma 
\ref{t1lem1}, there is a counterexample  $(O_ \delta ) _{ \delta \in  \beta ^* } $
 to $[ \beta ^*   \beta ^*  ]$-\brfrt compactness
such that each $O _{ \alpha _n} $ is indispensable.
Thus, if
$H \subseteq \beta ^*$ and  
 $(O_ \delta ) _{ \delta \in  \beta ^* } $ is a cover of $X$, 
then $H$ has order type $\beta ^*$, and
moreover
$\alpha_n \in H$, for every $ n \in \omega$.

Let $A= (\beta ^* + \omega) \setminus \{ \alpha _ n \mid n \in \omega \} $.
$A$ has order type $\beta ^* + \omega$, since
$\beta ^{*} $ is expressible as a limit of ordinals of uncountable cofinality,
hence taking off a sequence of order type $ \omega$ does not alter the order type
of $\beta ^{*} $.
Let $(O'_ \delta ) _{ \delta \in  A } $
be defined by
$O'_ \delta = O_ \delta $, if 
$ \delta \in \beta ^*  \setminus \{ \alpha _ n \mid n \in \omega \} $,
and by 
$O' _{ \beta ^* + n}  = O _{ \alpha _n} $,
for $ n \in \omega$.
Since these latter elements of the cover are indispensable, it is
easy to see that $(O'_ \delta ) _{ \delta \in  A } $ is a counterexample to 
$[ \beta ^* + \omega ,  \beta ^* + \omega  ]$-\brfrt compactness.
 \qedhere$_ {Claim\ 2}$ 
 \end{proof}

\noindent
 \emph{Proof of  Theorem    \ref{t1} (continued).} 
Summing up, we have proved the theorem  in the case when 
either 
  \begin{enumerate}    
\item $ \beta = \beta ^{*}+ \omega  $, or 
\item $ \beta = \beta ^{*} $ and $ \cf \beta ^{*} = \omega$.
 \end{enumerate} 

Now let $\beta$ be arbitrary.
By definition, $\beta \geq \beta ^{*}$, and, since we have assumed $ \cf \beta= \omega $,
we have further that,  if  $ \cf \beta ^{*} > \omega$, then
$ \beta \geq \beta ^{*}+ \omega  $. 
Notice also that, by definition, there is $\gamma$ with $| \gamma | \leq \omega $
such that  $ \beta = \beta ^{*}+ \gamma  $ and, if 
$ \cf \beta ^{*} > \omega$, then, by above, there is $\gamma'$ with 
$| \gamma' | \leq \omega $ such that $ \beta = \beta ^{*}+ \omega +\gamma'  $.

Now observe that, if the statement of the theorem  holds for some given ordinal $\beta'$ in place of $\beta$, and $\beta''$ is another ordinal such that $\beta''= \beta ' + \gamma $, for some $\gamma$ with
$| \gamma | \leq \omega $, then 
 the statement of the theorem  holds for $\beta''$ in place of $\beta$, too.

The above observations show that the two already proved particular cases 
(1) and (2) imply the statement of the theorem  in its full generality.
\end{proof} 

\begin{remark} \labbel{rmkcfw}   
(a) The assumption that $\beta$ has cofinality
$ \omega$ in Theorem  \ref{t1} is necessary.
By Example \ref{exex}(3),
if $\kappa$ is regular and uncountable,
then $(\kappa, \ord)$ 
is $[ \kappa +\omega , \kappa + \omega  ]$-\brfrt compact,
but not 
$[ \kappa , \kappa   ]$-\brfrt compact,
hence the implication
 (3)  $\Rightarrow $  (1)
 in the statement of Theorem \ref{t1}
 fails, for $\beta= \kappa $
and $\alpha= \omega $.  

(b) On the other hand, for $\beta\geq \omega $, and $T_1$ spaces, the implication
 (1)  $\Rightarrow $  (2) 
 in  Theorem \ref{t1}
always holds, even without the
assumption that $\beta$ has cofinality
$ \omega$.
Indeed, by Proposition \ref{basict1}(4),
$[ \beta  , \beta   ]$-\brfrt compactness
implies 
$[ \beta ^\ell , \beta ^\ell  ]$-\brfrt compactness, thus,
without loss of generality, we can suppose that $\beta$ 
is limit.
Then, for every $\alpha < \omega_1 $, we get 
$[ \beta + \alpha  , \beta + \alpha   ]$-\brfrt compactness:
 this follows from Theorem \ref{t1} itself, in case 
$\cf \beta= \omega $,
and from Corollary \ref{cortransfer}(3)
and Proposition \ref{simple}(1),
if $\cf \beta > \omega $.

(c) On the contrary, 
the implication
 (1)  $\Rightarrow $  (2) 
 in the statement of Theorem \ref{t1}
fails, in general, for non $T_1$ spaces.
See, for example, the first example in Remark 
\ref{w+ww2}, with $\kappa= \omega $.

(d) Also 
the implication
 (3)  $\Rightarrow $  (1) 
 in the statement of Theorem \ref{t1}
fails, in general, for non $T_1$ spaces. Just consider
Example \ref{exex}(2), and take 
$\beta= \kappa = \omega $ and arbitrary 
$\alpha>1$. 
 \end{remark}

\begin{corollary} \labbel{t1cor} 
Suppose that $X$ is  $T_1$. Then 
 $X$  
is $[ \omega , \omega  ]$-\brfrt compact if and only if
$X$  is $[ \alpha , \alpha   ]$-\brfrt compact, for some
(equivalently, every) countably infinite ordinal $\alpha$, if and only if
$X$ is $[ \omega , \omega_1  )$-\brfrt compact.
 \end{corollary} 

\begin{proof}
The corollary follows by taking $\beta= \omega $ in Theorem  \ref{t1}. 
 \end{proof}  

Theorem \ref{t1} can be used to strengthen Proposition \ref{basict1}.  

\begin{definition} \labbel{bestaast2}
Recall from Definition \ref{bestaast} the definition of $\beta^*$.
For an ordinal $\beta$, define $\beta ^{**} $ as follows:
\[ 
\beta ^{**} 
 =\begin{cases}
\beta ^*  &   \text{if either $ \cf \beta ^* = \omega  $, or $ \beta = \beta ^* + n$, for some $n < \omega $},\\
 \beta ^* + \omega   &  \text{otherwise}.\\
\end{cases}
\] 

Notice that $ \beta ^{**} \leq \beta $, for every ordinal $\beta$.
 \end{definition}

\begin{corollary} \labbel{cor**}
Suppose that  $X$ is  $T_1$, and $ \beta \leq \alpha $ are infinite ordinals. Then
the following conditions are equivalent.
\begin{enumerate}
\item
$X$  is  $[ \beta , \alpha  ]$-\brfrt compact.
\item
$X$  is  $[ \beta ^{**} , \alpha    + \omega _1 )$-\brfrt compact.
\item
$X$  is both  $[ \beta ^{**}, \beta ^{**} ]$-\brfrt compact, 
and 
  $[ \gamma  ,  \gamma ]$-\brfrt compact,
for every $\gamma$ such that 
$ \beta \leq \gamma \leq \alpha$ and 
$\gamma= \gamma ^*$.  
 \end{enumerate}  
 \end{corollary} 

\begin{proof}
(1) $\Rightarrow $  (3)
From Proposition \ref{simple}(1) we get  $[ \beta , \beta  ]$-\brfrt compactness.
If $ \cf  \beta ^{**}= \omega $, then  
$[ \beta ^{**}, \beta ^{**} ]$-\brfrt compactness
follows from Theorem \ref{t1}(3) $\Rightarrow $  (1), 
with  $\beta ^{**}$ in place of $\beta$, and  since, by the definitions
of  $\beta ^{*}$ and of $\beta ^{**}$, we have that 
$ \beta = \beta ^{**} + \alpha'$, for some $\alpha'$ with $| \alpha '| \leq \omega $.  
If $ \cf  \beta ^{**} \not = \omega $, then 
 $ \beta = \beta ^{**} + n$, for some $ n <\omega$, and  
$[ \beta ^{**}, \beta ^{**} ]$-\brfrt compactness follows from
Proposition \ref{basict1}(1), since $\beta$ is assumed to be infinite. 
Finally,
  $[ \gamma  ,  \gamma ]$-\brfrt compactness,
for every $\gamma$ such that 
$ \beta \leq \gamma \leq \alpha$,
is trivial, by Proposition \ref{simple}(1). 

In order to prove (3) $\Rightarrow $  (2), 
in view of Proposition \ref{simple}(4), 
it is enough to prove 
  $[ \varepsilon  ,  \varepsilon  ]$-\brfrt compactness,
for every $ \varepsilon $ such that 
$ \beta ^{**}  \leq \varepsilon  < \alpha + \omega _1$.
Let us fix some $\varepsilon$ as above, and let
$\gamma = \varepsilon ^*$. Notice that 
$\gamma = \gamma ^*$, and 
that $\gamma \leq \alpha $, since 
$|[ \alpha , \varepsilon ] | \leq \omega $.    
If $\gamma \geq \beta $,
then, by assumption, we have  
  $[ \gamma   ,  \gamma  ]$-\brfrt compactness,
which implies 
  $[ \varepsilon  ,  \varepsilon  ]$-\brfrt compactness,
by Theorem \ref{t1} and Corollary \ref{cortransfer}(3),
as remarked in Remark \ref{rmkcfw}(b).   
On the other hand, if $\gamma < \beta $,
then $\varepsilon^* = \beta ^*$, since 
$ \beta ^* \leq \beta ^{**}  \leq \varepsilon $,
and $\varepsilon ^* = \gamma < \beta $.
Then  $[ \beta ^{**}, \beta ^{**} ]$-\brfrt compactness
implies   $[ \varepsilon  ,  \varepsilon  ]$-\brfrt compactness,
again by Remark \ref{rmkcfw}(b).

(2) $\Rightarrow $  (1) follows from Proposition \ref{simple}(1), 
since $ \beta ^{**}  \leq \beta $. 
\end{proof}

In particular, the compactness properties of $T_1$ spaces are completely
determined by checking $[ \beta  , \beta  ]$-\brfrt compactness
for 
  \begin{enumerate}   
\item $\beta$ finite,
 \item $\beta= \omega $, 
  \item $ \beta $ of uncountable cofinality  
 \item $\beta= \gamma  + \omega $, for $ \gamma $ of uncountable cofinality, and
 \item $\beta$  of cofinality $ \omega$, but expressible as a limit of ordinals of uncountable cofinality.
 \end{enumerate}

The above statement, and the next corollary as well, follow
from Corollary \ref{cor**} (1) $\Rightarrow $  (3) and the fact that,
for infinite $\beta$, both $\beta ^{*} $  and $\beta ^{**} $ have necessarily
one among the forms (2)-(4).

\begin{corollary} \labbel{lindt1} 
If $X$ is $T_1$, and $\beta$ is  the  Lindel\"of ordinal of $X$, 
then $\beta$ has one of the above 
 forms (1)-(5). In particular,
if $\beta < \omega _1$, then $\beta \leq \omega $.  
\end{corollary}

\begin{remark} \labbel{uncountaredifferent}    
 It follows from Example \ref{exex}(3) that 
the behavior of countable ordinals 
in Theorem  \ref{t1} and Corollary \ref{t1cor}  
constitutes an exceptional case. The situation is radically different for larger cardinals and ordinals, even for normal topological spaces. Indeed,
if $\kappa$ is a regular and uncountable cardinal, then
$(\kappa, \ord)$ is 
$[ \kappa + \kappa , \kappa + \kappa    ]$-\brfrt compact but not
$[ \kappa , \kappa    ]$-\brfrt compact. Thus,
\ref{t1} and  \ref{t1cor} do not hold when $ \omega$ is
replaced by an uncountable cardinal.

As another example, 
the disjoint union of two copies of $(\kappa, \ord)$ is 
$[ \kappa + \kappa + \kappa , \kappa + \kappa  + \kappa   ]$-\brfrt compact, but not
$[ \kappa + \kappa , \kappa + \kappa    ]$-\brfrt compact (see Example \ref{samecard}).
\end{remark}

However, Theorem \ref{t1}  does admit a generalization to larger
cardinals, but only under a somewhat stronger assumption.

\begin{definition} \labbel{t1surrogatelambda} 
If $\lambda$ is an infinite cardinal, we say
 that 
$(X, \tau)$ is $\lambda$-$T_1$ if and only if, 
for every $O \in \tau$, and every $ Z \subseteq X$
with $| Z | < \lambda $,
$O \setminus  Z \in \tau  $.
Thus, $T_1$ is the same as $ \omega$-$T_1$.

If $(X, \tau)$ is a $T_1$ topological space, 
and the intersection of $<\lambda$ open sets of $X$
is still an open set of $X$, then
$(X, \tau)$ is
$\lambda$-$T_1$ in the above sense.
 \end{definition}   

\begin{proposition} \labbel{t1lambda}
Suppose that  $X$ is  $\lambda$-$T_1$, and $ \beta $ is a
limit  ordinal of
cofinality $ \leq \lambda $. Then
the following conditions are equivalent.
\begin{enumerate}
\item
$X$  is  $[ \beta , \beta   ]$-\brfrt compact.
\item
$X$  is  $[ \beta  +\alpha, \beta   +\alpha ]$-\brfrt compact, for
every ordinal  $\alpha$ with $|\alpha| \leq \lambda  $.  
\item
$X$  is  $[ \beta  +\alpha, \beta   +\alpha ]$-\brfrt compact, for
some ordinal  $\alpha$ with $|\alpha| \leq \lambda $.  
\item
$X$  is  $[ \beta ,  \beta   + \lambda ^+)$-\brfrt compact.
 \end{enumerate}  
\end{proposition}  

The next lemma is proved as Lemma \ref{t1lem1}.
 
\begin{lemma} \labbel{t1lemlam}
Suppose that $\lambda$ is an infinite cardinal, 
$\alpha$ and $ \gamma $ are limit ordinals,
 $  \gamma  \leq \lambda $, 
$ \cf \gamma  = \cf \alpha $,
and $(\alpha _ \zeta ) _{ \zeta  \in \gamma  } $
is a strictly increasing sequence 
such that
$\sup_{ \zeta  \in \gamma  } \alpha_ \zeta  = \alpha  $.

If $X$ is $\lambda$-$T_1$ and not 
$[ \alpha , \alpha ]$-\brfrt compact,
then there is a counterexample 
$(O_ \delta ) _{ \delta \in \alpha} $  
to the
$[ \alpha , \alpha ]$-\brfrt compactness of $X$ 
with the property that, for every $ \zeta  \in \gamma $,
$O _{ \zeta  } $ is indispensable.
\end{lemma} 

\begin{proof}[Proof of Proposition \ref{t1lambda}.]
If $\beta$ is any ordinal, let  $\beta ^{* \lambda } $
  be the smallest ordinal $ \leq \beta$  
such that $|[ \beta ^{* \lambda } , \beta ]| \leq \lambda $.
Thus, $\beta ^{* \lambda } $ is  the largest ordinal $ \leq \beta$
which is either $ 0$, or  has cofinality $>\lambda$,
or  can be written as a limit of ordinals of cofinality $>\lambda$.

The proof now follows the lines of the proof of 
Theorem  \ref{t1}: prove 
first the result
in the case when 
$ \beta = \beta ^{* \lambda }+ \lambda  $, and then
when $ \beta = \beta ^{* \lambda } $ and
$ \omega \leq cf \beta ^{* \lambda } \leq \lambda  $.
\end{proof}

 \section{Related notions and problems} \labbel{concl}

The spaces introduced in Examples \ref{exex}(3) and \ref{samecard} 
are normal topological spaces (with a base of clopen sets),
and they thus provide certain limits to provable results
for $[ \beta , \alpha ]$-compactness of normal spaces.
However, the theory developed so far appears to be not sharp 
enough to deal with such spaces.

As a very rough hypothesis, we conjecture that there is not
very much difference in the theory of
$[ \beta , \alpha ]$-compactness for, 
say, $T_1$ spaces and Tychonoff spaces.
We also conjecture that we can 
get some more theorems under the additional assumption of normality.
All the above rough hypotheses need to be verified; the present note
appears to be already long enough, thus we postpone the discussion
of such matters
to a subsequent work.

\begin{problem} \labbel{char} 
Give characterizations, similar to the ones given in Theorems
\ref{iff1} and \ref{iff2},  
for those pairs
of ordinals $\alpha$ and $\beta$ such that 
$[ \alpha , \alpha ]$-\brfrt compactness implies
$[ \beta , \beta ]$-\brfrt compactness,
for general topological spaces and, respectively,
for topological spaces satisfying some given separation axiom.
Of course, the spaces introduced in  Examples \ref{exex},
 \ref{exgen}, \ref{samecard},  \ref{condiscr}, as well as the spaces 
$X_ \tau ( \beta , \alpha )$ and $X_T ( \beta , \alpha )$
of Definitions \ref{saa}  will be relevant
to the solution of this problem.
\end{problem}

\begin{remark} \labbel{linlind} 
Indeed, for normal spaces, some problems might be open 
even restricted to cardinal compactness.
For example, it is easy to see that $X$ is a linearly 
Lindel\"of not Lindel\"of space (see \cite{AB})
if and only if $X$ is 
$[ \kappa  , \kappa  ]$-\brfrt compact,
for every regular uncountable cardinal $\kappa$, but
there is some uncountable cardinal $\lambda$ 
(necessarily  of cofinality $ \omega$) such that 
$X$ is not $[ \lambda , \lambda  ]$-\brfrt compact.
\end{remark}   

\begin{problem} \labbel{prod} 
Study the behavior of 
$[ \beta , \alpha ]$-\brfrt compactness of topological 
spaces with respect to products.

This problem might have some interest, since
 nontrivial results about cardinal compactness
of products of topological spaces are already known. 
See, e. g., \cite{Sto,GS,ScSt,Cprepr,C}.
See \cite{topproc,topappl} for further results and references.
\end{problem}

\begin{problem} \labbel{othercpn}
Study the mutual relationships among
$[ \beta , \alpha ]$-\brfrt compactness and other
compactness properties, either defined in terms
of covering properties or not.
 \end{problem}

\begin{definition} \labbel{oba}
We can also generalize the present notion 
of ordinal compactness to the relativized  notion introduced in
\cite{LiF}.

If $X$ is a topological space, and
$\mathcal F$ is a family of subsets of $X$,
let us say that $X$ is
$\mathcal F$-$ [ \beta , \alpha  ]$-\emph{compact}
if and only if the following condition holds.

For every sequence  $( C _ \delta  ) _{ \delta  \in \alpha   } $
 of closed sets of $X$, if,
for every $H \subseteq \alpha  $ with order type $< \beta $,
there exists $F \in \mathcal F$ such that   
$ \bigcap _{ \delta \in H}  C_ \alpha \supseteq F$,
then  
$ \bigcap _{ \delta \in \alpha  }  C_ \alpha \not= \emptyset $.

With this notation, 
$ [ \beta , \alpha  ]$-compactness
turns out to be the particular case of 
$\mathcal F$-$ [ \beta , \alpha  ]$-compactness
when $\mathcal F$ is the set of all singletons of $X$.

The particular case when $\mathcal F$ is the set of
all nonempty open sets of $X$ might have particular interest.
The corresponding  notion when 
both $\alpha$ and $\beta$ are cardinals has been 
studied in \cite{LO}. 

Still another generalization is suggested by 
\cite{LiF}. 
If
$\mathcal F$ is a family of subsets of $X$, let us say that $X$ is
$ [ \beta , \alpha  ]$-\emph{compact relative to} $\mathcal F$  
if and only if the following condition holds.

For every sequence  $( F _ \delta  ) _{ \delta  \in \alpha  } $
 of elements of $\mathcal F$, if,
for every $H \subseteq \alpha  $ of order type $ < \beta $, 
$ \bigcap _{ \delta \in H}  F_ \delta  \not= \emptyset $,
then  
$ \bigcap _{ \delta \in \alpha}  F_ \alpha \not= \emptyset $.
For a topological space $X$, 
$ [ \beta , \alpha  ]$-compactness is the same as
$ [ \beta , \alpha  ]$-compactness relative to the family
of all closed subsets of $X$.  
 \end{definition}

\begin{problem} \labbel{logics}
  A similar definition of ordinal compactness can 
be given for abstract logics. See \cite{E} for definitions
and background about logics.

Let us say that a logic $\mathcal L$ is $( \alpha , \beta )$-\emph{compact}   
if and only if, for every $\alpha$-indexed set 
$ (\sigma_ \delta ) _{ \delta \in \alpha } $
of $\mathcal L$-sentences, if, 
for every $H \subseteq \alpha$ with 
order type $<\beta$,   
$  \{ \sigma_ \delta \mid    \delta \in H \} $
has a model, then
$ \{ \sigma_ \delta \mid    \delta \in \alpha \}  $
has a model.

Notice the reversed order of $\alpha$ and $\beta$, to be 
consistent with the standard notation used in the literature about compactness
of logics.

We do not know whether ordinal compactness for logics is really
a new notion, that is, whether or not it can be expressed in terms
of cardinal compactness only. See, e. g., \cite{Ma} for notions of cardinal compactness 
for logics. 
 \end{problem}

The idea of defining  $[ \beta, \alpha]$-\brfrt compactness
came to us after reading the  definition of an $( \alpha , \kappa )$-regular
ultrafilter
 in \cite[p. 237]{BK}.

\begin{definition} \labbel{general}    
We can define an even more general notion of compactness.
If $Z$ is any set, and $W$ is a subset of the power set of $Z$, say 
that a topological space is \emph{$[W, Z]$-compact} if and only if, whenever
$(O_z) _{ z \in Z } $  is an open cover of $X$, then there is
$w \in W$ such that  
$(O_z) _{ z \in w } $  is still a cover of $X$.

The usual notion of $[\mu, \lambda ]$-compactness is the particular case
when $Z$ has cardinality $\lambda$, and $W$ is the set of all subsets 
of $Z$ of cardinality $<\mu$.

More generally, our notion of 
$[ \beta , \alpha  ]$-compactness is the particular case
when $Z= \alpha $, and $W$ is the set of all subsets 
of $ \alpha $ of order type $< \beta $.

We do not know whether there are other significant particular cases.
 \end{definition}

\end{document}